\documentclass[reqno, a4paper,11pt]{amsart}
\usepackage{amssymb,amsmath}
\usepackage{color}
\usepackage{graphicx}
\usepackage{paralist}

\usepackage{epstopdf}

\usepackage[colorlinks=true, linkcolor=black, citecolor=black]{hyperref}

\usepackage[top=1.5in, bottom=1.5in, left=1in, right=1in, headsep=0.5in, footskip=0.5in, a4paper]{geometry}

\usepackage{tikz}

\newcommand{\Gh}{\Gamma_h}
\newcommand{\Lin}{L^{\infty}}
\newcommand{\llin}{\ell^{\infty}}
\newcommand{\Pin}{P^{\infty}}
\newcommand{\ppin}{p^{\infty}}
\newcommand{\dt}{\partial_t}
\newcommand{\bdt}{\bar\partial_\tau}

\def\RR{\mathbb{R}}
\newtheorem{theorem}{Theorem}
\newtheorem{lemma}[theorem]{Lemma}
\newtheorem{problem}[theorem]{Problem}

\theoremstyle{definition}
\newtheorem{remark}[theorem]{Remark}

\begin{document}

\title[Analysis and solution of volume-surface reaction-diffusion systems]{Analysis and numerical solution of coupled \\volume-surface reaction-diffusion systems\\ with application to cell biology}

\author[H. Egger, K. Fellner, J.-F. Pietschmann, B.Q. Tang] {H. Egger, K. Fellner, J.-F. Pietschmann, B.Q. Tang}

\address{Herbert Egger, 
AG Numerical Analysis and Scientific Computing, 
Department of Mathematics, 
TU Darmstadt, 
Dolivostr. 15, 64293 Darmstadt, Germany.}
\email{egger@mathematik.tu-darmstadt.de}

\address{Klemens Fellner, 
Institute of Mathematics and Scientific Computing, 
University of Graz, 
Heinrichstra{\ss}e 36, 8010 Graz, Austria.}
\email{klemens.fellner@uni-graz.at}

\address{Jan-Frederik Pietschmann, 
Institute for Applied Mathematics,
Department of Mathematics and Computer Science, 
University of M\"unster, 
Orleansring 10, 48149 M\"unster, Germany.}
\email{pietschmann@mathematik.tu-darmstadt.de}

\address{Bao Q. Tang, 
Institute of Mathematics and Scientific Computing, 
University of Graz, 
Heinrichstra{\ss}e 36, 8010 Graz, Austria, and 
Faculty of Applied Mathematics and Informatics, 
Hanoi University of Science and Technology, 
1~Dai Co Viet, Hai Ba Trung, Hanoi, Vietnam.}
\email{quoc.tang@uni-graz.at} 



\begin{abstract}
We consider the numerical solution of coupled volume-surface reaction-diffusion systems having a detailed balance equilibrium.
Based on the conservation of mass, an appropriate quadratic entropy functional is identified and an entropy-entropy dissipation inequality is proven. This allows us to show exponential convergence to  equilibrium by the entropy method.
We then investigate the discretization of the system by a finite element method and an implicit time stepping scheme including the domain approximation by polyhedral meshes.
Mass conservation and exponential convergence to equilibrium are established on the discrete level by arguments similar to those on the continuous level 
and we obtain estimates of optimal order for the discretization error which hold uniformly in time. 
Some numerical tests are presented to illustrate these theoretical results.
The analysis and the numerical approximation are discussed in detail for a simple model problem. The basic arguments however apply also in a more general context. This is demonstrated by investigation of a particular volume-surface reaction-diffusion system arising as a mathematical model for asymmetric stem cell division. 
\end{abstract}

\maketitle

\section{Introduction} \label{sec:intro}

Various physical phenomena in biology, material science, or chemical engineering
are driven by reaction-diffusion processes in different compartments 
and by transfer between them. 
This may involve mass transfer between different domains but also with domain interfaces or boundaries.
In cell-biology, for instance, many phenomena are based on reaction-diffusion processes of proteins within the cell-body and on the cell cortex \cite{Mayor03, NGCRSS}. 
Particular examples are systems modeling cell-biological signaling processes \cite{FNR} or models for asymmetric stem cell division which describe the localization of so-called cell-fate determinants during mitosis \cite{BMK,MEBWK, WNK}.  

As a simple model problem for such coupled volume-surface reaction-diffusion processes, we consider the system
\begin{subequations}\label{eq:1}
\begin{align}
\dt L - d_L\Delta L = 0  &\qquad\text{on } \Omega, \label{eq:1a} \\
\dt \ell - d_{\ell}\Delta_{\Gamma}\ell = \lambda L - \gamma \ell &\qquad \text{on } \Gamma, \label{eq:1b} \\
d_L\partial_{\nu}L = -\lambda L + \gamma \ell &\qquad \text{on } \Gamma,   \label{eq:1c}
\end{align}
\end{subequations}
which describes the diffusion of a substance with concentration $L$ in the volume and concentration $\ell$ on the surrounding surface coupled by mass transfer between these two compartments.
Here $\Omega$ is the volume domain, $\Gamma=\partial\Omega$ is the surface, $\Delta_\Gamma$ denotes the Laplace-Beltrami operator, and the model parameters $d_L,d_\ell,\lambda,\gamma$ are assumed to be positive constants. 
The equations are supposed to hold for all $t>0$ and are complemented by appropriate initial conditions.

The system \eqref{eq:1a}--\eqref{eq:1c} may serve as a starting point for considering more realistic adsorption and desorption processes
or as a reduced model for volume-surface reaction-diffusion processes arising in cell biology. 
But despite its simplicity, this model problem already features some interesting properties which will be of main interest for our further considerations:
\begin{itemize} 
\item[(i)] The system preserves non-negativity of solutions.
\item[(ii)] The total mass $M=\int_\Omega L dx + \int_\Gamma \ell ds$ is conserved for all time. 
\item[(iii)] There exists a unique constant positive detailed balance equilibrium which can be parametrized explicitly by the total mass $M$ and the model parameters $\lambda$ and $\gamma$.
\item[(iv)] The solutions are uniformly bounded and converge exponentially fast towards the equilibrium state with respect to any reasonable Lebesgue or Sobolev norm. 
\end{itemize}
The first goal of our paper will be to establish these properties, in particular (iii) and (iv), for the model problem \eqref{eq:1a}--\eqref{eq:1c}.
Based on this analysis, the second aim is then to investigate the systematic discretization of the system by finite elements and implicit time-stepping schemes.
The guideline for the construction of the numerical approximation is to preserve and utilize the key features of the model as far as possible also for the analysis on the discrete level.
In particular, we will establish the conservation of mass, prove the existence of a unique discrete equilibrium, and show the exponential convergence to equilibrium on the discrete level. We also comment briefly on the possibility for preserving the non-negativity of solutions.

Our arguments are closely related to the concepts of structure-preserving schemes, geometric integration, and compatible discretization that have been developed over the last decades. 
We refer the reader to the extensive survey \cite{CMO11} and the references therein.  
Several results concerning the preservation of entropic structures for discretization schemes have been obtained more recently, see e.g. \cite{BEJ14,CG15,CJS15,GG09,JS15}. 
This paper adds to this research by addressing the analysis and numerical solution of coupled volume-surface reaction-diffusion systems.

%
Apart from the entropy arguments, the numerical analysis is based on standard tools for the discretization of evolution problems \cite{Thomee,Wheeler73}, on previous work concerning the analysis of domain approximations \cite{BaDeck99,Deck,Deck_etal09,DE07}, and on recent results from \cite{ER}, who considered a somewhat simpler stationary volume-surface reaction-diffusion problem. 
The entropy estimates will allow us to establish order optimal convergence estimates with constants that are uniform in time. 

The model problem \eqref{eq:1a}--\eqref{eq:1c} is simple enough to avoid complicated notation and to present our basic ideas in the most convenient way to the reader. 
To illustrate the applicability to more general problems, we consider in Section~\ref{sec:extension} the following system 
which models the evolution of four conformations of the key protein Lgl during the mitosis 
of Drosophila SOP precursor stem-cells \cite{BMK,FRT,MEBWK}.
\begin{subequations}\label{eq:2}
\begin{align}
\dt L - d_L\Delta L = -\beta L + \alpha P &\quad \text{ on } \Omega, \label{eq:2a}\\
\dt P - d_P\Delta P = \beta L - \alpha P, &\quad \text{ on }\Omega, \label{eq:2b}\\
\dt \ell - d_{\ell}\Delta_{\Gamma}\ell = -d_L\partial_n L + \chi_{\Gamma_2}(-\sigma \ell + \kappa p) &\quad \text{ on }\Gamma, \label{eq:2c}\\
\dt p - d_p\Delta_{\Gamma} = \sigma \ell - \kappa p - d_P\partial_n P &\quad \text{ on }\Gamma_2.\label{eq:2d}
\end{align}
Note that the last equation only holds on a part $\Gamma_2\subseteq \Gamma$ of the boundary.
The mass transfer between the domain and the surface is described by
\begin{align}
d_L\partial_n L = -\lambda L + \gamma \ell &\quad \text{ on }\Gamma, \label{eq:2e}\\
d_P\partial_n P = \chi_{\Gamma_2}(-\eta P + \xi p) &\quad \text{ on }\Gamma, \label{eq:2f}\\
d_p\partial_{n_{\Gamma}}p = 0 &\quad \text{ on }\partial\Gamma_2.\label{eq:2g}
\end{align}
\end{subequations}
The diffusion and reaction parameters are positive constants, and the equations are 
assumed to hold for $t>0$ and are complemented by appropriate initial conditions. 
In the cell-biological context, $L$ and $\ell$ denote the concentrations of native Lgl
within the cell cytoplasm and on the cell cortex, while $P$ and $p$ denote the corresponding phosphorylated Lgl conformations.

Let us point out that like the model problem \eqref{eq:1a}--\eqref{eq:1c}, also the system \eqref{eq:2a}--\eqref{eq:2g} involves fully reversible reactions and mass transfer processes;
see Figure~\ref{fig:dynamics} for a schematic sketch.
\begin{figure}[!ht]
\begin{center}
\begin{tikzpicture}
\node (x) at (0,0) {$L$}  node (y) at (0,-2) {$\ell$};
\draw[arrows=->] ([xshift =0.5mm,yshift=0.5mm]y.north) -- node [right] {\scalebox{.8}[.8]{$\gamma$}} ([xshift =0.5mm,yshift=-0.5mm]x.south);
\draw[arrows=->] ([xshift =-0.5mm,yshift=-0.5mm]x.south) -- node [left] {\scalebox{.8}[.8]{$\lambda$}} ([xshift =-0.5mm,yshift=0.5mm]y.north);
\end{tikzpicture}
\qquad\qquad
\begin{tikzpicture}
\node (a) at (0,0) {$L$} node (b) at (2,0) {$P$} node (c) at (0,-2) {$\ell$} node (d) at (2,-2)  {$p$};
\draw[arrows=->] ([xshift =0.5mm,yshift=0.5mm]a.east) -- node [below] {\scalebox{.8}[.8]{$\alpha$}} ([xshift =-0.5mm,yshift=0.5mm]b.west);
\draw[arrows=->] ([xshift =-0.5mm,yshift=-0.5mm]b.west) -- node [above] {\scalebox{.8}[.8]{$\beta$}} ([xshift =0.5mm,yshift=-0.5mm]a.east);
\draw[arrows=->] ([xshift =0.5mm,yshift=0.5mm]c.north) -- node [right] {\scalebox{.8}[.8]{$\gamma$}} ([xshift =0.5mm,yshift=-0.5mm]a.south);
\draw[arrows=->] ([xshift =-0.5mm,yshift=-0.5mm]a.south) -- node [left] {\scalebox{.8}[.8]{$\lambda$}} ([xshift =-0.5mm,yshift=0.5mm]c.north);
\draw[arrows=->] ([xshift =0.5mm, yshift=-0.5mm]c.east) -- node [below] {\scalebox{.8}[.8]{$\sigma$}} ([xshift=-0.5mm, yshift=-0.5mm]d.west);
\draw[arrows=->] ([xshift =-0.5mm, yshift=0.5mm]d.west) -- node [above] {\scalebox{.8}[.8]{$\kappa$}} ([xshift=0.5mm, yshift=0.5mm]c.east);
\draw[arrows=->] ([xshift = 0.5mm, yshift =0.5mm]d.north) -- node [right] {\scalebox{.8}[.8]{$\xi$}} ([xshift = 0.5mm, yshift=-0.5mm]b.south);
\draw[arrows=->] ([xshift = -0.5mm, yshift =-0.5mm]b.south) -- node [left] {\scalebox{.8}[.8]{$\eta$}} ([xshift = -0.5mm, yshift=0.5mm]d.north);
\end{tikzpicture} 
 \caption{Mass transfer dynamics for model \eqref{eq:1a}--\eqref{eq:1c} (left) and the cell-biologically inspired model \eqref{eq:2a}--\eqref{eq:2g} (right).
The top line represents the volume concentrations and the bottom line those on the surface. All reactions and mass transfer processes are reversible with constant reaction coefficients. }\label{fig:dynamics}
\end{center}
\end{figure}
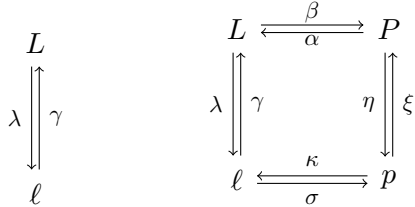

\noindent
We will show that the key properties (i)--(iv) also hold for the system \eqref{eq:2a}--\eqref{eq:2g}, i.e. solutions remain non-negative, the total mass is conserved, there exists a unique positive detailed balance equilibrium, and the concentrations converge exponentially fast to equilibrium. These assertions will be verified by minor modifications of the arguments already used for the analysis of the model problem \eqref{eq:1a}--\eqref{eq:1c}.
As a consequence, also the discretization strategy can be extended in a straight forward manner yielding approximations of optimal order and uniform in time.

\medskip 

The remainder of the manuscript is organized as follows:
Section~\ref{sec:analysis} is devoted to the analysis of problem~\eqref{eq:1a}--\eqref{eq:1c} on the continuous level. 
%
%
In Section~\ref{sec:notation}, we introduce the basic notation for the finite element discretization and recall some results about domain approximations. 
In Section~\ref{sec:semi}, 
we analyze the semi-discretization in space.
%
%
Section~\ref{sec:time} is concerned with the time discretization by the implicit Euler method. 
%
%
The validity of the theoretical results is illustrated by some numerical tests in Section~\ref{sec:num},
and in Section~\ref{sec:extension} we show how to generalize our arguments to the model for asymmetric stem cell division mentioned above.
%
The presentation closes with a short summary of our results and a discussion of open problems.


\section{Analysis of the model problem} \label{sec:analysis}

Let us start with introducing the relevant notation and basic assumptions and give a precise 
statement of our model problem. 

\subsection{Preliminaries}
Let $\Omega \subset\RR^n$, $n=2$ or $3$, be a bounded domain with sufficiently smooth boundary, say $\Gamma=\partial\Omega \in C^3$.
We consider the volume-surface reaction-diffusion system
\begin{subequations}\label{eq:3}
\begin{align}
\dt L - d_L \Delta L = 0, & \qquad x \in \Omega, \ t > 0, \label{eq:3a} \\
\dt \ell - d_\ell \Delta_\Gamma \ell = \lambda L - \gamma \ell, & \qquad x \in \Gamma, \ t>0, \label{eq:3b} 
\end{align}
subject to the interface and initial conditions
\begin{align}
d_L \partial_n L = \gamma \ell - \lambda L, \qquad &  x \in \Gamma, \ t>0, \label{eq:3c} \\
L(0)=L^0, \quad x \in \Omega,  \qquad 
&\ell(0)=\ell^0,  \quad x \in \Gamma. \label{eq:3d} 
\end{align}
\end{subequations}
The diffusion coefficients $d_L$, $d_\ell$ and the reaction constants $\lambda$, $\gamma$ are assumed to be positive constants.
We use standard notation for function spaces \cite{Evans} and we denote by $(\cdot,\cdot)_\Omega$ and $(\cdot,\cdot)_{\Gamma}$ the inner products and by $\|\cdot\|_\Omega$ and $\|\cdot\|_{\Gamma}$ the induced norms of $L^2(\Omega)$ and $L^2(\Gamma)$, respectively. 

\subsection{Existence, uniqueness, and regularity}

Let us briefly discuss the existence and uniqueness of weak solutions and their qualitative properties which we shall need later on.

\begin{problem}[Weak formulation]\label{prob:weak}
Let $T > 0$ be fixed. 
Find 
$$
(L,\ell)\in C([0,T];L^2(\Omega) \times L^2(\Gamma)) \cap L^2(0,T;H^1(\Omega) \times H^1(\Gamma))
$$ 
such that $L(0)=L^0$, $\ell(0)=\ell^0$, and 
\begin{align}\label{eq:weak}  
c(\dt L(t),\dt \ell(t);v,w)+ a(L(t),\ell(t);v,w) = 0
\end{align}
for all $v \in H^1(\Omega)$ and $w \in H^1(\Gamma)$, and for a.e. $0<t<T$. 
The bilinear forms are given by
\begin{align} \label{eq:a}
a(L,\ell;v,w) 
&:=  d_L (\nabla L, \nabla v)_{\Omega} 
  +  d_{\ell} (\nabla_{\Gamma} \ell, \nabla_{\Gamma} w)_{\Gamma} 
  + ( \lambda L - \gamma \ell, v - w)_{\Gamma},
\end{align} 
and by $c(L,\ell;v,w)=(L,v)_\Omega+(\ell,w)_\Gamma$, respectively. 
\end{problem}
%
%
This problem fits into the abstract framework of parabolic evolution equations as 
discussed, for instance, in \cite{DautrayLions5,Evans} and can be analyzed with standard arguments.
A main ingredient is 

\begin{lemma}[G\r{a}rding inequality] \label{lem:garding}
For all  $v \in H^1(\Omega)$ and $w \in H^1(\Gamma)$ there holds
$$
  a(v,w;v,w) \ge 
 \|v\|^2_{H^1(\Omega} + \| w\|^2_{H^1(\Gamma)} - \eta (\|v\|^2_\Omega + \|w\|_{\Gamma}^2)
$$ 
with a constant $\eta>0$ depending only on the parameters $d_L$, $d_\ell$, $\gamma$, $\lambda$, and on the domain $\Omega$.
\end{lemma}
\begin{proof}
The result follows from the definition of $a(\cdot;\cdot)$ and the Cauchy-Schwarz inequality. 
\end{proof}
Let us emphasize that the bilinear form $a(\cdot;\cdot)$ is not elliptic which can easily be seen by choosing $v$ and $w$ as appropriate constants. Using the G\r{a}rding inequality and standard arguments for parabolic evolution problems, one readily obtains

\begin{lemma}[Well-posedness of the weak formulation] \label{lem:existence}
For any pair of initial data $L^0 \in L^2(\Omega)$ and $\ell^0 \in L^2(\Gamma)$, the problem \eqref{eq:3a}-\eqref{eq:3d} has a unique weak solution 
$(L,\ell)$, and
$$
\sup_{t\in (0,T)} (\|L(t)\|^2_\Omega  +\|\ell(t)\|^2_\Gamma) + \int_0^T (\|L(t)\|_{H^1(\Omega)}^2  +\|\ell(t)\|_{H^1(\Gamma)}^2)\,dt
\le C_T (\|L^0\|^2_{\Omega} + \|\ell^0\|^2_{\Gamma}).
$$
If the initial data are non-negative, then the solution remains non-negative for all time.
\end{lemma}
\begin{proof}
Existence and uniqueness of a global weak solution follow by standard results, see e.g., \cite[Chapter XVIII]{DautrayLions5}, and positivity can be established by an iteration argument.
\end{proof}

From the abstract analysis one obtains a-priori estimates with a constant $C_T$ that eventually increases with $T$. We will later show by entropy arguments that the bounds are actually uniform in time. 
%
A basic ingredient for our analysis will be the fact that the total mass is conserved during the evolution of our system. 

\begin{lemma}[Mass conservation] \label{lem:properties}
Let $(L,\ell)$ denote a weak solution of \eqref{eq:3a}--\eqref{eq:3d} and 
denote by 
\begin{align*} 
M(t)=(L(t),1)_\Omega + (\ell(t),1)_{\Gamma} 
\end{align*}
the total mass at time $t$. Then, 
$$
M(t)=M(0)=:M^0
\quad \text{for all}\quad t>0.
$$
\end{lemma}
\begin{proof}
Testing the weak form \eqref{eq:weak} with $v \equiv 1$ and $w \equiv 1$, we get 
\begin{align*}
\tfrac{d}{dt} M(t) 
= \tfrac{d}{dt} (L(t), 1)_\Omega + \tfrac{d}{dt} (\ell(t),1)_{\Gamma} 
= c(\dt L(t),\dt \ell(t);1,1)
= -a(L(t),\ell(t);1,1)  = 0,
\end{align*}
and the result follows by integration with respect to time.
\end{proof}

\subsection{Equilibrium system}

A reaction-diffusion system like \eqref{eq:3a}--\eqref{eq:3c} can be assumed to eventually tend to equilibrium on the long run. For the problem under consideration, the equilibrium concentrations $L^\infty$, $\ell^\infty$ satisfy the stationary system
\begin{subequations} \label{eq:5}
\begin{align}
- d_L \Delta L^\infty = 0 \qquad \qquad \ \; &\qquad \text{in } \Omega, \label{eq:5a}\\
- d_\ell \Delta_\Gamma \ell^\infty = \lambda L^\infty - \gamma \ell^\infty & \qquad \text{on } \Gamma, \label{eq:5b}  \\
d_L \partial_n L^\infty = \gamma \ell^\infty - \lambda L^\infty & \qquad \text{on } \Gamma. \label{eq:5c}
\end{align}
Because of the mass conservation, we additionally know that 
\begin{align} \label{eq:5d}
(L^\infty,1)_\Omega + (\ell^\infty,1)_{\Gamma} &= M^0, 
\end{align}
\end{subequations}
where $M^0=\int_\Omega L^0 + \int_\Gamma \ell^0$ is the total initial mass of the system.
Note that this extra condition is already required to ensure the uniqueness of the equilibrium state. 
Using the above notation, the weak form of the equilibrium system is given by
\begin{problem}[Equilibrium system] \label{prob:equilibrium}
Find $L^\infty \in H^1(\Omega)$ and $\ell^\infty \in H^1(\Gamma)$ 
such that
\begin{align}\label{eq:statweak}
a(L^\infty,\ell^\infty ; v,w) = 0,
\end{align}
for all $v \in H^1(\Omega)$ and $w \in H^1(\Gamma)$ and such that the mass constraint \eqref{eq:5d} is satisfied.
\end{problem}

For showing well-posedness of the equilibrium problem, we will utilize the following 
identity.
\begin{lemma}[Inf-sup stability]
For any $v \in H^1(\Omega)$ and $w \in H^1(\Gamma)$ there holds 
\begin{align}\label{eq:infsup}
a(v,w;\lambda v,\gamma w) 
&= \lambda d_L \|\nabla v\|_{\Omega}^2 + \gamma d_\ell \|\nabla_{\Gamma} w\|_{\Gamma} + \|\lambda v - \gamma w\|^2_{\Gamma}.
\end{align} 
\end{lemma}
As we will show next, the right hand side actually defines a norm on the space of functions with zero total mass. 
The following Poincar\'{e}-type inequality will also play a crucial role for the analysis of the time-dependent problem.
\begin{lemma}[Poincar\'{e}-type inequality] \label{lem:poincare}
There exists a constant $C_P>0$ depending only on the parameters $d_L$, $d_\ell$, $\lambda$, $\gamma$, and on the domain $\Omega$, such that 
\begin{align}\label{eq:poincare}
\lambda \|L\|^2_{H^1(\Omega)} + \gamma \|\ell\|^2_{H^1(\Gamma)}
\le C_P \left( \lambda d_L\| \nabla L\|^2_\Omega + \gamma d_\ell \|\nabla_\Gamma \ell\|^2_{\Gamma} + \|\lambda L - \gamma \ell\|^2_{\Gamma}\right)
\end{align}
for all $L \in H^1(\Omega)$ and $\ell \in H^1(\Gamma)$ satisfying the mass constraint $(L,1)_\Omega + (\ell,1)_{\Gamma}=0$.
\end{lemma}
\begin{proof}
The right hand side is zero if, and only if, $L$ and $\ell$ are constant and 
$L=\frac{\gamma}{\lambda} \ell$. Since the parameters $\gamma,\lambda$ are positive, the mass constraint yields $L=\ell=0$. Therefore, the term in parenthesis on the right hand side of \eqref{eq:poincare} defines a norm on $H^1(\Omega)\times H^1(\Gamma)$. 
The uniform estimate is obtained by the \emph{lemma of equivalent norms}; see e.g. \cite[Ch.~11]{Tartar}. 
\end{proof}

A combination of the previous two lemmas already yields the well-posedness of the equilibrium problem. Since the right hand side in \eqref{eq:statweak} is zero, the solution can however even be obtained explicitly here. We will make use of the particular structure of the equilibrium later on.

\begin{lemma}[Equilibrium] \label{lemma:equilibrium}
The system \eqref{eq:5a}--\eqref{eq:5d} has a unique weak solution $(L^\infty,\ell^\infty)$ given by
\begin{align} \label{eq:formulas}
L^\infty = \frac{\gamma M^0}{\gamma |\Omega| + \lambda |\Gamma|}
\qquad \text{and} \qquad 
\ell^\infty 
= \frac{\lambda}{\gamma} L^\infty.
\end{align}
\end{lemma}
\begin{proof}
One easily verifies that $(L^\infty,\ell^\infty)$ given by the formulas above is a solution 
of \eqref{eq:5a}--\eqref{eq:5c} satisfying the mass constraint \eqref{eq:5d}.
Now assume that $(L^*,\ell^*)$ is any other weak solution to \eqref{eq:5a}--\eqref{eq:5d}. 
Then, the difference $(L^{\infty} - L^*, \ell^{\infty} -\ell^*)$ has zero total mass, and 
\begin{align}
a(L^{\infty} - L^*, \ell^{\infty} -\ell^*; v, w) = 0
\end{align}
for all $v\in H^1(\Omega)$ and $w\in H^1(\Gamma)$. By \eqref{eq:infsup} we thus get
\begin{align}
\lambda d_L  \|\nabla (L^{\infty} - L^*)\|^2_\Omega + \gamma d_\ell \|\nabla_\Gamma (\ell^{\infty}-\ell^*)\|_{\Gamma}^2 + \|\lambda (L^{\infty} - L^*)- \gamma (\ell^{\infty}-\ell^*) \|_{\Gamma}^2 = 0,
\end{align}
and Lemma \ref{lem:poincare} implies that $L^{\infty} = L^*$ and $\ell^{\infty} = \ell^*$. 
This shows the uniqueness.
\end{proof}

\subsection{Convergence to equilibrium}

We will now show that the solution $(L(t),\ell(t))$ converges to the equilibrium $(L^\infty,\ell^\infty)$ exponentially fast by using the entropy method. 
For a given constant equilibrium state $(L^\infty,\ell^\infty)$, we define the 
quadratic relative entropy functional
\begin{align}\label{eq:entropy}
E(L,\ell) = \frac{1}{2} \left( \lambda \|L-L^\infty\|_\Omega^2 + \gamma \|\ell-\ell^\infty\|_\Gamma^2 \right),
\end{align}
which is the square of a scaled $L^2$-distance to the equilibrium in the product space $L^2(\Omega) \times L^2(\Gamma)$.

\begin{lemma}[Entropy dissipation]\label{lem:entropy}
Let $(L, \ell)$ denote a weak solution of problem \eqref{eq:3a}--\eqref{eq:3d} with corresponding constant equilibrium state $(L^\infty,\ell^\infty)$. 
Then 
\begin{align}\label{entropydissipation}
\tfrac{d}{dt}E(L(t),\ell(t))
& = -d_L \lambda \|\nabla (L(t)-L^\infty)\|_\Omega^2 
- d_\ell \gamma \|\nabla_\Gamma (\ell(t) - \ell^\infty)\|^2_{\Gamma}\\
&\qquad  \qquad - \|\lambda (L(t) - L^\infty) - \gamma (\ell(t) - \ell^\infty)\|^2_{\Gamma}=:-D(L(t),\ell(t)) \notag
\end{align}
for all $t>0$. The functional $D(\cdot)$ is called the \emph{entropy dissipation}.
\end{lemma}
Let us emphasize that the entropy $E$ and the entropy dissipation $D$  depend on the equilibrium states $L^\infty$ and $\ell^\infty$. For ease of notation, we do not write this explicitly.
\begin{proof}
By definition of the entropy functional $E(\cdot)$ and elementary manipulations, we obtain
\begin{align*}
\tfrac{d}{dt} E(L(t),\ell(t))
&= \lambda (\dt L(t), L(t)-L^\infty)_\Omega + \gamma (\dt \ell(t), \ell(t)-\ell^\infty)_{\Gamma} 
\\&
=-a(L(t)-L^\infty,\ell(t)-\ell^\infty;\lambda (L(t)-L^\infty),\gamma (\ell(t)-\ell^\infty)) \\
&= - d_L \lambda \|\nabla (L-L^\infty)\|_\Omega^2 
- d_\ell \gamma \|\nabla_\Gamma (\ell - \ell^\infty)\|^2_{\Gamma} - \|\lambda (L - L^\infty) - \gamma (\ell - \ell^\infty)\|^2_{\Gamma},
\end{align*}
which already yields the desired result.
\end{proof}

From the Poincar\'e-type inequality established in Lemma~\ref{lem:poincare}, we may directly deduce
\begin{lemma}[Entropy-entropy dissipation inequality] \label{lem:eed}
For all $L \in H^1(\Omega)$ and $\ell \in H^1(\Gamma)$ satisfying the mass constraint $(L,1)_\Omega+(\ell,1)_\Gamma=M^0$, there holds
\begin{align} \label{eq:eed}
D(L, \ell)\geq c_0 E(L, \ell) \qquad \text{with} \quad c_0=2/C_P. 
\end{align}
\end{lemma}

A combination of the previous estimates now immediately yields

\begin{theorem}[Exponential stability]\label{thrm:equilibrium}
Let $(L,\ell)$ be the weak solution of the system \eqref{eq:3a}--\eqref{eq:3c} and 
$(L^\infty,\ell^\infty)$ denote the corresponding equilibrium.
Then, 
\begin{align*}  
\|L(t) - L^\infty\|^2_\Omega + \|\ell(t) - \ell^\infty\|^2_{\Gamma} \le C e^{- c_0 t} \left( \|L^0 - L^\infty\|^2_\Omega + \|\ell^0 - \ell^\infty\|^2_{\Gamma}\right)  
\end{align*}
for all $t > 0$ with constants $c_0,C>0$ depending only  $d_L$, $d_\ell$, $\lambda,\gamma$,
and on the domain $\Omega$. 
\end{theorem}
\begin{proof}
As a direct consequence of Lemma~\ref{lem:entropy} and \ref{lem:eed}, we get 
\begin{align*}
\tfrac{d}{dt} E(L(t),\ell(t)) \le -c_0 E(L(t),\ell(t)) \quad \text{for all } \quad t > 0.
\end{align*}
Therefore, the classical Gronwall inequality gives 
$$
E(L(t),\ell(t)) \le e^{-c_0 t} E(L^0,\ell^0).
$$ 
The desired result then follows from the fact that $\lambda$, $\gamma$ are positive constants and since the entropy $E(L,\ell)$ is equivalent to the square of the $L^2$-norm distance to equilibrium.
\end{proof}

\begin{remark}
The dependence of the constant $c_0$ on the parameters $\lambda$, $\gamma$, $d_L$, $d_\ell$, and on certain geometric constants can be made explicit, even for some nonlinear problems; see e.g. \cite{FPT}. For linear problems, $c_0$ can also be determined from a generalized eigenvalue problem. 
\end{remark}

\begin{remark}
As a consequence of Theorem~\ref{thrm:equilibrium}, we obtain 
uniform bounds for the solution in $L^\infty(0,\infty;L^2(\Omega) \times L^2(\Gamma))$. 
Since the problem \eqref{eq:3a}--\eqref{eq:3c} is linear and the coefficients are independent of time, one can obtain uniform bounds and even exponential decay also for $(\dt^j L, \dt^j \ell)$, $j \ge 0$ in $L^p(0,\infty;H^k(\Omega) \times H^k(\Gamma))$ provided that certain compatibility conditions hold; we refer to \cite[Section~7]{Evans} for some basic results in this direction.
\end{remark}

\section{Basic notation and domain approximations} \label{sec:notation}

We now introduce the notation needed for the formulation of the finite element approximation of our problem  and recall some basic results about the domain approximation by polyhedral meshes from \cite{Deck,ER}.
For ease of presentation, we will only consider the two dimensional case here. 
All arguments however easily generalize to dimension three.

\subsection{The mesh and domain mappings}
Let $\Omega \subset \RR^2$ be a bounded domain with smooth boundary $\Gamma \in C^3$.
We approximate $\Omega$ by a polygonal domain $\Omega_h$ for which a 
conforming triangulation $T_h=\{T\}$ is available. As usual, we denote by $\rho_T$ and $h_T$ the incircle radius and the diameter of the triangle $T$, respectively, and we call 
$h=\max_T h_T$ the mesh size. We further denote by $E_h=\{e \}$ the partition of the domain boundary $\Gamma$ into edges $e$ inherited from the triangulation $T_h$.
Throughout the subsequent sections, we make use of the following assumptions. 

\medskip 

\noindent
(A1) $T_h$ is $\gamma$-shape-regular, i.e., there exists a constant $\gamma>0$ such that 
\begin{align*} 
\rho_T \le h_T \le \gamma \rho_T, \qquad \text{for all } T \in T_h. 
\end{align*}

\noindent 
(A2) There exists a diffeomorphism $G_h : \Omega_h \to \Omega$ such that the estimates
\begin{align*}
&\|G_h - id\|_{L^\infty(T)} \leq \beta h_T^2, 
&
&\|D G_h - I\|_{L^\infty(T)} \le \beta h_T,
&
&\|\det (D G_h) - 1\|_{L^\infty(T)} \le \beta h_T,
\\
&\|D G_h^{-1}\|_{L^\infty(G_h(T))} \le \beta,  
&
&\|D^2 G_h\|_{L^\infty(T)} \le \beta, 
&
&\|D^2 G_h^{-1}\|_{L^\infty(G_h(T))} \le \beta
\end{align*}
for all $T \in T_h$ and $G_h(x)=x$ for all elements $T$ with $G_h(T) \cap \Gamma=\emptyset$.

\smallskip 
\noindent 
(A3) The induced surface mapping $g_h : \Gamma_h \to \Gamma$ defined by $g_h = G_h |_{\Gamma_h}$ in addition satisfies
\begin{align*}
\|\det(Dg_h)-1\|_{L^\infty(e)} \le \beta h_T^2, \qquad e \in E_h, \ e \subset \partial T. 
\end{align*}

\smallskip
We will consider families of partitions $\{T_h\}_{h>0}$ later on,
and the constants $\gamma,\beta$ are then assumed to be independent of $T_h$,
in particular, of the mesh size. 
An explicit construction of appropriate domain mappings $G_h$ can be found in \cite[Section 4.2]{ER}; see also \cite{Deck_etal09,DE07,Scott73}.

\subsection{Restriction to the discrete domain}
To compare functions defined on $\Omega$ and $\Omega_h$,
we associate to any $u : \Omega \to \RR$ a function
\begin{align} \label{eq:lifting}
\widetilde u := u \circ G_h 
\end{align}
defined on $\Omega_h$ called the \emph{restriction} of $u$ to $\Omega_h$.
Using the above properties of the domain mapping $G_h$, one easily verifies that
\begin{align}\label{eq:Resequi1}
c \|\widetilde u\|_{H^k(T)} \le \|u\|_{H^k(G_h(T))} \le c^{-1} \|\widetilde u\|_{H^k(T)}, 
\end{align}
for all $u \in H^k(G_h(T))$ and $k \le 2$ with a positive constant $c$ that only depends on $\beta$.
%
%
In a similar manner, we define for any function $p : \Gamma \to \RR$ defined on the boundary $\Gamma$ the restriction $\widetilde p$ to the boundary $\Gamma_h$ of the discrete domain by
\begin{align}\label{eq:liftsurface}
\widetilde p := p \circ g_h.
\end{align}
By the chain rule and application of the previous estimates, we readily obtain
\begin{align}\label{eq:Resequi2}
c\|\widetilde p\|_{H^k(e)} \le \|p\|_{H^k(g_h(e))} \le c^{-1} \|\widetilde p\|_{H^k(e)}
\end{align}
for all functions $p \in H^k(e)$ with $e \in E_h$ and integers $k \le 2$.

\section{Semi-discretization in space}\label{sec:semi}

We now investigate the semi-discretization of problem \eqref{eq:3a}--\eqref{eq:3c} in space.
To this end, let 
$$
V_h = \{ v_h \in C(\Omega_h): v_h|_T \in P_1(T) \quad \text{for all } T \in T_h\}
$$ 
be the space of continuous piecewise linear functions over $\Omega_h$, and denote by 
$$
W_h =  \{ w_h : \Gamma_h \to \RR : w_h = v_h|_{\Gamma_h} \quad \text{for some } v_h \in V_h \}
$$ 
the corresponding space of continuous piecewise linear functions on the surface $\Gamma_h$. 
Note that by construction $V_h \subset H^1(\Omega_h)$ and $W_h \subset H^1(\Gamma_h)$. 
Moreover, $W_h$ contains all traces of functions in $V_h$ which we will use later on. 

\subsection{Finite element discretization of the evolution problem}

As approximation of the volume-surface reaction-diffusion system 
defined by \eqref{eq:3a}--\eqref{eq:3d}, we consider 

\begin{problem}[Semi-discretization] \label{prob:1} 
Let $(L^0,\ell^0) \in L^2(\Omega)$ and set $\widetilde L^0=L^0 \circ G_h$ and $\widetilde \ell^0 = \ell^0 \circ g_h$. 
Find $(L_h,\ell_h) \in H^1(0,T;V_h \times W_h)$ such that
\begin{align*}
(L_h(0),v_h)_{\Omega_h}=(\widetilde L^0,v_h)_{\Omega_h} 
\qquad  \text{and} \qquad 
(\ell_h(0),w_h)_{\Gamma_h} = (\widetilde \ell^0,w_h)_{\Gamma_h},
\end{align*}
for all $v_h \in V_h$, $w_h \in W_h$ and such that 
\begin{align} \label{eq:discrete}
c_h(\dt L_h(t),\dt \ell_h(t), v_h,w_h) + a_h(L_h,\ell_h;v_h,w_h) &= 0 
\end{align}
holds for all $v_h \in V_h$ and $w_h \in W_h$, and for all $0 < t \le T$.
The bilinear forms are given by 
\begin{align}\label{disbilinear}
a_h(L,\ell;v,w)  = d_L (\nabla L, \nabla v)_{\Omega_h} + d_\ell (\nabla_{\Gamma_h} \ell, \nabla_{\Gamma_h} w)_{\Gamma_h}  + (\lambda L - \gamma \ell, v-w)_{\Gamma_h}.
\end{align}
and $c_h(L,\ell;v,w)=(L,v)_{\Omega_h}+(\ell,w)_{\Gamma_h}$, respectively.
\end{problem}

Note that the variational principle \eqref{eq:discrete} has the same structure as the weak form  \eqref{eq:weak}  of the continuous problem which allows us to utilize similar arguments as in 
Section~\ref{sec:analysis} for the analysis of the discrete problem.
By choice of a basis for the finite dimensional spaces $V_h$ and $W_h$, 
Problem~\ref{prob:1} can be recast as a linear system of ordinary differential equations.
Existence and uniqueness of a solution then follow immediately from the Picard-Lindel\"of theorem. 
\begin{lemma} \label{lem:existenceh}
For any $L^0 \in L^2(\Omega)$ and $\ell^0 \in L^2(\Gamma)$, Problem~\ref{prob:1} has a unique solution.
\end{lemma} 

As a next step, let us note that the total mass is conserved also on the discrete level. 
Choosing $v_h \equiv 1$ and $w_h \equiv 1$ as test functions in \eqref{eq:discrete}, we 
obtain similar to the continuous level

\begin{lemma}[Mass conservation] \label{lem:conservationh}
Let $(L_h,\ell_h)$ denote the solution of Problem~\ref{prob:1}. Then 
\begin{align*}
(L_h(t),1)_{\Omega_h} + (\ell_h(t),1)_{\Gamma_h} = (\widetilde L^0,1)_{\Omega_h} + (\widetilde \ell^0,1)_{\Gamma_h} =: M^0_h, 
\end{align*}
for all $t>0$, i.e., the total mass is conserved for all times also on the discrete level. 
\end{lemma}

\begin{remark}
With the usual mass lumping strategy, one could also 
preserve the non-negativity of the semi-discrete solution provided that the 
triangulation satisfies somewhat stronger conditions. 
We will not go into details here, but refer the reader to \cite{ChenThomee85,Thomee} for details. 
\end{remark}

\subsection{Geometric errors}

As mentioned before, we can use the domain mappings $G_h : \Omega_h \to \Omega$ and $g_h : \Gamma_h \to \Gamma$ to define restrictions 
$$
\widetilde L = L \circ G_h \qquad \text{and} \qquad \widetilde \ell = \ell \circ g_h,
$$
of the continuous solution onto the discrete domains which can then be compared with the discrete solution. 
Some properties of the restrictions have already been discussed in Section~\ref{sec:notation}.
For the numerical analysis, we now proceed with similar arguments as in \cite{Deck_etal09,ER}. 
However, we will work most of the time on the discrete domain $\Omega_h$ here instead of $\Omega$.  
To this end, let us also define the restriction of the bilinear form $a(\cdot;\cdot)$ to the discrete domain $\Omega_h$ by
\begin{align}\label{restrict_a}
\widetilde  a(\widetilde L,\widetilde \ell; \widetilde v,\widetilde w):= a(L,\ell; v, w) \quad \text{ for all } \quad (L,\ell), (v,w)\in H^1(\Omega)\times H^1(\Gamma).
\end{align}
Using the transformation formulas for integrals and derivatives, we can express $\widetilde  a(\cdot;\cdot)$  by 
\begin{align*}
\widetilde  a(\widetilde L,\widetilde \ell; \widetilde v,\widetilde w)
= d_L ( A \nabla \widetilde L, \nabla \widetilde v)_{\Omega_h} + d_\ell (B \nabla_{\Gh} \widetilde \ell, \nabla_{\Gh} \widetilde w)_{\Gamma_h} + (C (\lambda \widetilde L - \gamma \widetilde \ell), \widetilde v - \widetilde w)_{\Gamma_h}
\end{align*}
with coefficient functions
\begin{align*}
A = (D G_h^\top D G_h)^{-1} \det (D G_h), 
\quad 
B = (D g_h^\top D g_h)^{-1} \det (D g_h),
\quad \text{and} \quad 
C = \det (D g_h).
\end{align*}
The bilinear form $a_h(\cdot;\cdot)$ used to define the finite element approximations can therefore be considered to be a non-conforming approximation of the bilinear form $\widetilde a(\cdot;\cdot)$.
In a similar manner, we define the restriction of the bilinear form $c(\cdot;\cdot)$ by $\widetilde c(\widetilde L,\widetilde \ell;\widetilde v,\widetilde w):= c(L,\ell;v,w)$ which can also be expressed explicitly by
\begin{align*}
\widetilde c(\widetilde L,\widetilde \ell;\widetilde v,\widetilde w)  
:= 
(K \widetilde L, \widetilde v)_{\Omega_h} + (C \widetilde \ell,\widetilde w)_{\Gamma_h}
\end{align*}
with $K=\det(DG_h)$ and $C=\det(Dg_h)$ as before. 
The variational characterization \eqref{eq:weak} of the continuous solutions can then be written equivalently as
\begin{align}\label{prob:4}
\widetilde c(\dt \widetilde  L(t), \dt \widetilde \ell(t); \widetilde v,\widetilde w) + \widetilde  a(\widetilde  L(t), \widetilde \ell(t); \widetilde v, \widetilde w) = 0 
\end{align}
for all $(\widetilde v,\widetilde w)\in H^1(\Omega_h)\times H^1(\Gamma_h)$, whereas the discrete variational problem \eqref{eq:discrete} reads
\begin{align}\label{prob:4h}
c_h(\dt L_h(t),\dt \ell_h(t); v_h,w_h) + a_h(L_h(t), \ell_h(t); v_h,w_h) = 0  
\end{align}
for all $(v_h,w_h) \in V_h \times W_h$. 
Note that the semi-discretization \eqref{prob:4h} can thus be interpreted as a non-conforming approximation of problem \eqref{prob:4} on the discrete domain.
The differences between the bilinear forms $a_h(\cdot;\cdot)$ and $\widetilde a(\cdot;\cdot)$ as well as $c_h(\cdot;\cdot)$ and $\widetilde c(\cdot;\cdot)$ are only due to \emph{geometric errors} that can be quantified explicitly
with similar arguments as in \cite[Lemma 6.2]{ER}.
\begin{lemma}\label{lem_a}
For all $(L,\ell)$ and $(v,w)\in H^1(\Omega)\times H^1(\Gamma)$ there holds 
\begin{align*} 
|a_h(\widetilde L,\widetilde \ell;\widetilde v,\widetilde w) - \widetilde  a(\widetilde L,\widetilde \ell; \widetilde v,\widetilde w) |  
&\le Ch\|(L,\ell)\|_{H^1(\Omega)\times H^1(\Gamma)}\|(\widetilde v,\widetilde w)\|_{H^1(\Omega_h)\times H^1(\Gamma_h)}.
\end{align*}
%
The constant $C$ is independent of the mesh size $h$. 
%
\end{lemma}
\begin{proof}
%
The estimates follow from \cite[Lemma 6.2]{ER} with minor modifications in the proofs. 
\end{proof}
\noindent 
We will later in particular also need the following estimate for the error in the bilinear form $c$. 
\begin{lemma} \label{lem:c}
For all $(L,\ell)$ and $(v,w)\in L^2(\Omega)\times L^2(\Gamma)$ there holds 
\begin{align*} 
|c_h(\widetilde L,\widetilde \ell;\widetilde v,\widetilde w) - \widetilde  c(\widetilde L,\widetilde \ell; \widetilde v,\widetilde w) |  
&\le C h\|(L,\ell)\|_{L^2(\Omega)\times L^2(\Gamma)}\|(\widetilde v,\widetilde w)\|_{L^2(\Omega_h)\times L^2(\Gamma_h)}.
\end{align*}
If in addition $L$, $v\in H^1(\Omega)$ and $\ell,w \in H^1(\Gamma)$, then 
\begin{align*} 
|c_h(\widetilde L,\widetilde \ell;\widetilde v,\widetilde w) - \widetilde  c(\widetilde L,\widetilde \ell; \widetilde v,\widetilde w) |
& \le C h^2 \|(L,\ell)\|_{H^1(\Omega) \times H^1(\Gamma)} \|(\widetilde v,\widetilde w)\|_{H^1(\Omega_h) \times H^1(\Gamma_h)}. 
\end{align*}
The constant $C$ in these estimates is again independent of the mesh size $h$. 
\end{lemma}
\noindent 
As will become clear from the proof, the regularity requirements could be somewhat relaxed.
\begin{proof}
By definition of the bilinear forms, we have 
\begin{align*}
c_h(\widetilde L,\widetilde \ell;\widetilde v,\widetilde w) - \widetilde  c(\widetilde L,\widetilde \ell; \widetilde v,\widetilde w)
&= \int_{B_h} \widetilde L \widetilde v (1-\det(DG_h)) dx + \int_{\Gamma_h} \widetilde \ell \widetilde w (1-\det(Dg_h)) ds. 
\end{align*}
with $B_h = \bigcup_{T \cap \Gamma_h \neq \emptyset} T$. 
Here we used that $G_h = id$ for element not adjacent to the boundary due to assumption (A2). 
Using the bounds for the determinants in (A2) and (A3), we get
\begin{align} \label{eq:estc} 
c_h(\widetilde L,\widetilde \ell;\widetilde v,\widetilde w) - \widetilde  c(\widetilde L,\widetilde \ell; \widetilde v,\widetilde w)
&\le C \big( h \|\widetilde L\|_{L^2(B_h)} \|\widetilde v\|_{L^2(B_h)} + h^2 \|\widetilde \ell\|_{L^2(\Gamma_h)} \|\widetilde w\|_{L^2(\Gamma)} \big).
\end{align}
This already implies the first bound. 
A localized Poincar\'e inequality similar to \cite[Lemma~4.10]{ER} 
and application of the trace theorem yields
\begin{align*}
\|\widetilde L\|_{L^2(B_h)}
&\le C \big ( h^{1/2} \|\widetilde L\|_{L^2(\Gamma_h)} + h \|\nabla \widetilde L\|_{L^2(B_h)}\big) 
\le C h^{1/2} \|\widetilde L\|_{H^1(\Omega_h)}.
\end{align*}
Using this and \eqref{eq:Resequi1} to bound the first term in \eqref{eq:estc} yields the second assertion.
\end{proof}

\subsection{The discretization of the equilibrium problem}

For the approximation of the equilibrium system \eqref{eq:5a}--\eqref{eq:5d}, 
we consider the following discrete variational problem.
\begin{problem}[Discrete equilibrium problem] \label{prob:3}
Find $L^\infty_h \in V_h$ and $\ell^\infty_h \in W_h$ satisfying the mass constraint $(L^\infty_h,1)_{\Omega_h}+(\ell^\infty_h,1)_{\Gamma_h}=M^0_h$ and such that for all $(v_h,w_h) \in V_h \times W_h$ there holds
\begin{align} \label{eq:equilibriumh}
a_h(L^\infty_h,\ell^\infty_h;v_h,w_h) = 0.
\end{align}
\end{problem}
In order to ensure the well-posedness of this problem, it suffices to show uniqueness,
which readily follows from the following two auxiliary results.
\begin{lemma}[Discrete inf-sup stability]
For any $v \in H^1(\Omega_h)$ and $w \in H^1(\Gamma_h)$ 
there holds
\begin{align}\label{eq:discinfsup}
a_h(v,w;\lambda v,\gamma w) 
&= \lambda d_L \|\nabla v\|_{\Omega_h}^2 + \gamma d_\ell \|\nabla_{\Gamma_h} w\|_{\Gamma_h}^2 + \|\lambda v - \gamma w\|^2_{\Gamma_h}.
\end{align}
\end{lemma}
%
%
The right hand side of \eqref{eq:discinfsup} defines a norm on the subspace of functions with zero total mass which follows by a Poincar\'e-type inequality. It will be important later on that the 
equivalence constant can be chosen to be independent of the mesh.
\begin{lemma} \label{lem:poincareh}
For any $v \in H^1(\Omega_h)$ and $w \in H^1(\Gamma_h)$ with $(v,1)_{\Omega_h} + (w,1)_{\Gamma_h}=0$, there holds
\begin{align} \label{eq:infsupstability}
\lambda \|L\|_{H^1(\Omega_h)}^2 + \gamma \|\ell\|_{H^1(\Gamma_h)}^2
\le C_P \; \big(\lambda d_L \|\nabla L\|_{\Omega_h}^2 + \gamma d_\ell \|\nabla_\Gamma \ell\|_{\Gamma_h}^2 + \|\lambda L - \gamma \ell\|^2_{\Gamma_h}\big)
\end{align}
with a constant $C_P$ that only depends on the parameters $d_L$, $d_\ell$, $\lambda$, and $\gamma$, 
on the domain $\Omega$, and the constants $\beta,\gamma$ in assumptions (A1)--(A3) but is otherwise independent of the mesh $T_h$.
\end{lemma} 
\begin{proof}
By transformation with the domain mapping $G_h$, all integrals in the definition of the norm can be cast into integrals over $\Omega$ and $\Gamma$, respectively. 
The assertion then follows from the Poincar\'e-type inequality stated in Lemma~\ref{lem:poincare}, the estimates for the geometric errors given in the previous section, and the norm equivalence estimates \eqref{eq:Resequi1} and \eqref{eq:Resequi2}.
\end{proof}

Since the constant $C_P$ can be chosen independently of the mesh size, 
we deliberately use the same symbol as on the continuous level here. 
By similar arguments as in Lemma \ref{lemma:equilibrium}, we obtain

\begin{lemma}[Discrete equilibrium] \label{lem:wellposedness}
Problem~\ref{prob:3} has a unique solution given by 
\begin{align}\label{eq:equilibriumhform}
L_h^{\infty} = \frac{\gamma M^0_h}{\gamma|\Omega_h|+\lambda|\Gamma_h|}
\qquad \text{and} \qquad 
\ell_h^{\infty} 
=\frac{\lambda}{\gamma} L^\infty_h.
\end{align}
\end{lemma}

The error analysis for the equilibrium problem could be obtained by a careful extension of the results in \cite{ER}. 
Since the equilibrium is constant on the continuous and on the discrete level, we can 
however give also a direct proof of the discretization error estimate here.
\begin{lemma} \label{lem:eeeq}
Let $(L^\infty,\ell^\infty)$ and $(L^\infty_h,\ell^\infty_h)$ be the solutions of \eqref{eq:statweak} and \eqref{eq:equilibriumh}, respectively
Then 
\begin{align*}
|L^\infty - L^\infty_h| + |\ell^\infty - \ell^\infty_h|
\le C h^2.
\end{align*}
with a constant $C$ that is independent of the mesh size. 
\end{lemma}
\begin{proof}
Using the explicit forms of $L^{\infty}$ and $L^{\infty}_h$ in \eqref{eq:formulas} and \eqref{eq:equilibriumhform} we can write
\begin{align*} 
\begin{aligned}
|L^{\infty} - L^{\infty}_h| &= \left|\frac{\gamma M^0}{\lambda|\Omega| + \gamma|\Gamma|} - \frac{\gamma M^0_h}{\lambda|\Omega_h| + \gamma|\Gamma_h|}\right|\\
&\leq c\left(\lambda|\Omega||M^0 - M^0_h| + \gamma M^0\big| |\Omega| - |\Omega_h|\big|\right)
\end{aligned}
\end{align*}
The result for the volume term then follows from the fact that the solutions are constant and from the properties of $G_h$ stated in assumption (A2). 
The estimates for the boundary component follow in a similar manner. 
\end{proof}

\subsection{Convergence to the discrete equilibrium}
Following the basic steps of the analysis on the continuous level, we 
next define the discrete entropy functional 
\begin{align}\label{entropyh}
E_h(L_h,\ell_h) = \frac{1}{2}\left(\lambda \|L_h-L^\infty_h\|_{\Omega_h}^2 +  \gamma \|\ell_h-\ell^\infty_h\|_{\Gamma_h}^2\right).
\end{align}
With the same arguments as in the proof of Lemma~\ref{lem:entropy}, we obtain
\begin{lemma}[Discrete entropy dissipation]\label{lem:disentropy}
Let $(L_h, \ell_h)$ denote the solution of discrete evolution Problem~\ref{prob:1} and $(L_h^\infty,\ell_h^\infty)$ be the corresponding discrete equilibrium.
Then,
\begin{align}\label{entropydissipationh}
\tfrac{d}{dt}E_h(L_h(t),\ell_h(t))
& = - d_L \lambda \|\nabla (L_h(t)-L^\infty_h)\|_{\Omega_h}^2 
- d_\ell \gamma \|\nabla_\Gamma (\ell_h(t) - \ell^\infty_h)\|^2_{\Gamma_h}\\
&\qquad \qquad - \|\lambda (L_h(t) - L^\infty_h) - \gamma (\ell_h(t) - \ell^\infty_h)\|^2_{\Gamma_h} =: -D_h(L_h(t),\ell_h(t)).  \notag
\end{align}
\end{lemma}

By the discrete Poincar\'e-type inequality stated in Lemma~\ref{lem:poincareh} we further get

\begin{lemma}[Discrete entropy-entropy dissipation inequality] \label{lem:diseed}
For any $v \in H^1(\Omega_h)$ and every $w \in H^1(\Gamma_h)$ satisfying  $(v,1)_{\Omega_h} + (w,1)_{\Gamma_h}=M_h^0$, there holds
\begin{align}\label{eq:disceed}
D_h(v,w) \ge c_0 E_h(v,w) \qquad \text{with} \quad c_0=2/C_P. 
\end{align}
\end{lemma}
Note that $c_0$ can be chosen independently of $h$.
Using the previous estimates, a Gronwall inequality, and the fact that the entropy is a 
scaled $L^2$-norm distance, we finally obtain
\begin{theorem}[Discrete exponential stability]\label{thrm:disconverge}
Let $(L_h,\ell_h)$ denote the solution to Problem~\ref{prob:1} and 
$(L^{\infty}_h,\ell^{\infty}_h)$ be defined as in Lemma~\ref{lem:wellposedness}. 
Then 
\begin{align}\label{disconverge}
\|L_h(t) - \Lin_h\|_{\Omega_h}^2 + \|\ell_h(t) - \llin_h\|_{\Gamma_h}^2 \leq Ce^{-c_0 t}\left(\|\widetilde L^0 - \Lin_h\|_{\Omega_h}^2 + \|\widetilde \ell^0 - \llin_h\|_{\Gamma_h}^2\right),
\end{align}
with constants $C>0$ and $c_0>0$ that are independent of $t$ and $h$.
\end{theorem} 

Let us emphasize that up to perturbations which vanish with $h \to 0$, 
the constants $C$ and $c_0$ can be chosen the same as on the continuous level.
The decay to equilibrium on the discrete level therefore occurs uniformly with respect o the mesh size and at the same rate as on the continuous level. 
This is also what we observe in our numerical tests; see Section~\ref{sec:num} for details.

\subsection{Error estimates for the semi-discretization}

The error analysis for the Galerkin approximation now proceeds with standard arguments \cite{Thomee,Wheeler73} but carefully taking into account the additional geometric errors. 
A key step is the definition of an appropriate operator $R_h : H^1(\Omega) \times H^1(\Gamma) \to V_h \times W_h$ that allows one to approximate the continuous solution $(L,\ell)$ by a function in  the discretization space $V_h \times W_h$.
For our purpose, we define $R_h$ by
\begin{align} \label{eq:Rh}
a_h(R_h ( L, \ell); v_h,w_h) + \eta c_h(R_h(L,\ell); v_h,w_h) 
= 
\widetilde a(\widetilde L, \widetilde \ell; v_h,w_h) 
+ \eta c_h(\widetilde L, \widetilde \ell; v_h,w_h) 
\end{align}
for all $(v_h,w_h)\in V_h\times W_h$. 
Note that the bilinear form $c_h$ appears on both sides in this definition.
Following standard convention, we call $R_h$ \emph{elliptic projection},  
although it is not a projection in the strict sense. 
The basic properties of $R_h$ needed later on can be summarized as follows.

\begin{lemma}[Elliptic projection]\label{lem:Ritz}
Let $\eta>0$ be chosen large enough. 
Then the elliptic projection $R_h: H^1(\Omega)\times H^1(\Gamma)\rightarrow V_h\times W_h$ is a well-defined bounded linear operator and the estimates 
\begin{align*}
\|(\widetilde L, \widetilde \ell) - R_h( L, \ell)\|_{H^1(\Omega_h)\times H^1(\Gamma_h)} \leq Ch\|(L,\ell)\|_{H^2(\Omega)\times H^2(\Gamma)}
\end{align*}
as well as
\begin{align*}
\|(\widetilde L, \widetilde \ell) - R_h(L,\ell)\|_{L^2(\Omega_h)\times L^2(\Gamma_h)} \leq Ch^2\|(L,\ell)\|_{H^2(\Omega)\times H^2(\Gamma)}
\end{align*}
hold for all $L \in H^2(\Omega)$ and $\ell \in H^2(\Gamma)$ with a constant $C$ independent of the mesh size. 
Moreover, 
\begin{align*}
c_h(R_h(L,\ell);1,1)=c_h(\widetilde L,\widetilde \ell; 1,1).
\end{align*} 
This implies that the elliptic projection $R_h$ is mass preserving in a generalized sense.
\end{lemma}
\begin{proof}
By Lemma~\ref{lem:garding}, the bilinear forms on both sides of \eqref{eq:Rh} are elliptic
for $\eta$ large enough. The elliptic projection therefore is a finite element approximation of an elliptic volume-surface reaction-diffusion problem. 
The error estimates then follow from the results in \cite{ER} with minor modification of the proofs. The conservation of mass is a direct consequence of the special construction of the elliptic projection and the fact that $a_h(\cdot,\cdot;1,1)=0$ and $\widetilde a(\cdot,\cdot;1,1)=0$.
\end{proof}

We are now in the position to state and prove our first main result.
\begin{theorem}[Error estimate for the semi-discretization] \label{thrm:convergencerate}
Let (A1)--(A3) hold and  assume that the solution of problem \eqref{eq:3a}--\eqref{eq:3d} is sufficiently smooth.
Then for all $t>0$ we have 
\begin{align*}
\|\widetilde L(t) - L_h(t)\|_{\Omega_h} + \|\widetilde \ell(t) - \ell_h(t)\|_{\Gamma_h} \leq Ch^2
\end{align*}
with a constant $C$ 
that is independent of $t$ and $h$.
\end{theorem}
\begin{proof}

Set $U = (L,\ell)$, $\widetilde U = (\widetilde L, \widetilde \ell)$, $U_h = (L_h,\ell_h)$, and $R_h U = (\widetilde L_R, \widetilde \ell_R)$.
%
%
In the usual manner, we decompose the error into
\begin{align*}
U_h(t) - \widetilde  U(t) = [\widetilde U(t) - R_h U(t)] + [U_h(t) - R_h U(t)]  =: \widetilde \rho(t) + \theta_h(t).
\end{align*}
By application of Lemma \ref{lem:Ritz}, we already obtain the desired estimate for the first component
\begin{align*}
\|\widetilde \rho(t)\|_{L^2(\Omega_h) \times L^2(\Gamma_h)} \leq Ch^2\|\widetilde  U(t)\|_{H^2(\Omega_h)\times H^2(\Gamma_h)}.
\end{align*}
To estimate the second term $\theta_h(t)$, we set $\Phi_h = (v_h,w_h)$ and note that 
\begin{align*}
c_h(\dt \theta_h(t);& \Phi_h)  + a_h(\theta_h(t); \Phi_h ) \\
&= c_h(\dt U_h(t); \Phi_h ) - c_h(R_h \dt U(t); \Phi_h ) 
    + a_h(U_h(t); \Phi_h ) - a_h(R_h U(t); \Phi_h )\\
&= \widetilde c(\dt \widetilde U(t); \Phi_h ) - c_h(R_h \dt U(t); \Phi_h ) 
    + \widetilde a(\widetilde U(t); \Phi_h ) - a_h(R_h U(t); \Phi_h ).
\end{align*}
Here we used the variational characterizations \eqref{prob:4} and \eqref{prob:4h} of the discrete and the continuous solution.
By definition \eqref{eq:Rh} of the elliptic projection, we further get
\begin{align*}
&c_h(\dt \theta_h(t); \Phi_h)  + a_h(\theta_h(t); \Phi_h ) \\
&= \widetilde c(\dt \widetilde U(t); \Phi_h ) - c_h(R_h \dt U(t); \Phi_h ) 
   - \eta c_h(\widetilde U(t);\Phi_h) + \eta c_h(R_h U(t);\Phi_h)\\
&= [\widetilde c(\dt \widetilde U(t); \Phi_h ) - c_h(\dt \widetilde U(t); \Phi_h )]
  + c_h(\dt \widetilde U(t) - R_h \dt U(t); \Phi_h )
  + \eta c_h(R_h U(t) - \widetilde U(t); \Phi_h ).
\end{align*}
Using Lemma \ref{lem:c} to bound the first term, the estimates for the elliptic projection given in Lemma \ref{lem:Ritz}, and the Cauchy-Schwarz inequality, we obtain
\begin{align*}
c_h(\dt \theta_h(t);&\Phi_h)  + a_h(\theta_h(t); \Phi_h ) \\
& \le C h^2 \big( \|U(t)\|_{H^2(\Omega) \times H^2(\Gamma)} + \|\dt U(t)\|_{H^2(\Omega) \times H^2(\Gamma)} \big) \|\Phi_h\|_{H^1(\Omega_h) \times H^1(\Gamma_h)}.
\end{align*}
%
We now choose $\Phi_h=(v_h,w_h)$ with $v_h=\lambda (L_h(t) - \widetilde L_R(t))$ and $w_h = \gamma(\ell_h(t) - \widetilde{\ell}_R(t))$ as the test function
and define $E_h(\theta(t)) = \frac{1}{2} \|\theta(t)\|_{L^2(\Omega_h) \times L^2(\Gamma_h)}^2$ and $D_h(\theta_h(t))=a_h(\theta_h(t);\Phi_h)$. 
With the previous estimates and the same reasoning as in Lemma~\ref{lem:disentropy} and \ref{lem:diseed}, we get 
\begin{align*}
&\tfrac{d}{dt}E_h(\theta_h(t)) + D_h(\theta_h(t)) 
=c_h(\dt \theta_h(t), \Phi_h ) + a_h(\theta_h(t); \Phi_h ) \\
& \qquad 
  \leq C' h^2 \big( \|\dt U(t)\|_{H^2(\Omega) \times H^2(\Gamma)} + \|U(t)\|_{H^2(\Omega) \times H^2(\Gamma)} \big) \|\theta_h(t)\|_{H^1(\Omega_h) \times H^1(\Gamma_h)}.
\end{align*}
We also used here that the norms of  $\Phi_h$ and $\theta_h(t)$ are comparable by construction. 
By Lemma~\ref{lem:Ritz} the total mass of $\theta_h(t) = U_h(t) - R_h U(t)$ is zero, and we conclude from Lemma \ref{lem:poincareh} that
\begin{align*}
D_h(\theta_h) \geq c \|\theta_h\|_{H^1(\Omega_h)\times H^1(\Gamma_h)}^2  \geq c' E_h(\theta).
\end{align*}
Thus we can absorb the term $\|\theta_h(t)\|_{H^1(\Omega^h) \times H^1(\Gamma_h)}$ in the above estimate via Young's inequality by the entropy dissipation $D_h(\theta_h(t))$ on the left hand side which finally leads to
\begin{align*}
\tfrac{d}{dt} E_h(\theta(t)) + c_1 E_h(\theta(t)) 
\leq C''h^4 \big (\|\dt U(t)\|_{H^2(\Omega) \times H^2(\Gamma)}^2 + \|U(t)\|_{H^2(\Omega)\times H^2(\Gamma)}^2 \big)
\end{align*}
The claim now follows by application of the Gronwall inequality and noting that 
the initial conditions are approximated with optimal order.
\end{proof}

\section{Time discretization} \label{sec:time}

As a final step towards the numerical solution, let us now turn to the 
time discretization. For ease of notation, we only consider the backward Euler scheme 
in detail here. After introducing the scheme and the relevant notation, we again establish the convergence to equilibrium and prove convergence estimates that are uniform in time. 

\subsection{Implicit backward Euler scheme}

Let $\tau>0$ be fixed and define $t^n = n\tau$ for all $n = 0,1,2,\ldots$. 
Given a sequence $\{u^n\}_{n \ge 0}$, we denote by
\begin{align} \label{BackQuotient}
\bdt u^n := \frac{u^{n} - u^{n-1}}{\tau} 
\end{align}
the backward difference quotient which serves as the approximation for the time derivative.
Our fully discrete approximation for the system \eqref{eq:3a}--\eqref{eq:3d} then reads 
\begin{problem}[Time discretization]\label{Prob3}
Define $(L_h^0, \ell_h^0)\in V_h\times W_h$ by
\begin{align}\label{DisInitial}
(L^0_h, v_h)_{\Omega_h} = (\widetilde L^0, v_h)_{\Omega_h} \quad \text{ and } \quad (\ell^0_h, w_h)_{\Gamma_h} = (\widetilde \ell^0, w_h)_{\Gamma_h}
\end{align}
for all $(v_h,w_h)\in V_h\times W_h$. 
Then for $n=1,2,\ldots$, find 
$(L^n_h, \ell^n_h)\in V_h\times W_h$ such that
\begin{align}\label{FullDis}
c_h(\bdt L_h^n,\bdt \ell_h^n; v_h,w_h)+ a_h(L_h^n, \ell_h^n; v_h, w_h) = 0
\end{align}
for all test functions $(v_h,w_h)\in V_h\times W_h$.
\end{problem}
The problems \eqref{FullDis} for the individual time steps can be written equivalently as
\begin{align}\label{VarForm_1}
\frac{1}{\tau}c_h(L_h^n,\ell_h^n; v_h,w_h)
+ a_{h}(L_h^n, \ell_h^n;v_h,w_h) = \frac{1}{\tau} c_h(L_h^{n-1},\ell_h^{n-1}; v_h,w_h).
\end{align}
From this representation and the discrete inf-sup stability condition \eqref{eq:discinfsup}, 
one can directly deduce that the problems \eqref{FullDis} are uniquely solvable.
\begin{lemma}[Well-posedness]
Let $(L_h^{n-1},\ell_h^{n-1}) \in L^2(\Omega_h) \times L^2(\Gamma_h)$ be given. \\
Then for any $\tau>0$, Problem~\ref{Prob3} admits a unique solution $(L_h^n, \ell_h^n) \in V_h \times W_h$.
\end{lemma}

Problem~\ref{Prob3} therefore defines a sequence $\{L_h^n,\ell_h^n\}_{n \ge 0} \subset V_h \times W_h$ which we call the fully discrete solution. By testing with $v_h \equiv 1$ and $w_h\equiv 1$, we further obtain

\begin{lemma}[Mass conservation]\label{lem:fullmass}
Let $M_h^0:= (\widetilde L^0,1)_{\Omega} + (\widetilde \ell^0,1)_{\Gamma_h}$. 
Then, 
\begin{align}\label{DisMassConser}
(L_h^n, 1)_{\Omega_h} + (\ell_h^n, 1)_{\Gamma_h} = M_h^0\qquad \text{for all } n \ge 1,
\end{align}
i.e., the total mass is conserved for all time steps.
\end{lemma}

\begin{remark}
By the usual mass lumping and under additional assumptions on the mesh, also 
non-negativity of solutions could be guaranteed for all time steps provided that the initial 
conditions were non-negative; let us again refer to \cite{ChenThomee85,Thomee} for details.
\end{remark}

\subsection{Convergence to the discrete equilibrium}

To study the large time behavior of the discrete evolution, 
we utilize the discrete entropy functional defined in Section~\ref{sec:semi} given by
\begin{align}\label{entropyht}
E_h(L,\ell) = \frac{1}{2}\left( \lambda \|L-L^\infty_h\|_{\Omega_h}^2 + \gamma \|\ell-\ell^\infty_h\|^2_{\Gamma_h}\right).
\end{align}
As a replacement for the entropy dissipation derived in Lemma~\ref{lem:disentropy}, 
we now have
\begin{lemma}[Fully discrete entropy dissipation]\label{lem:entropydiscrete}
For all $n\geq 1$, there holds
\begin{align}\label{en-endi-flux}
\bdt E_h(L_h^n,\ell^n_h) 
& \le - \lambda d_L \|\nabla (L_h^n-L^\infty_h)\|_{\Omega_h}^2 
- \gamma d_\ell \|\nabla_\Gamma (\ell_h^n - \ell^\infty_h)\|^2_{\Gamma_h}\\
&\quad \qquad \qquad - \|\lambda (L_h^n - L^\infty_h) - \gamma (\ell_h^n - \ell^\infty_h)\|^2_{\Gamma}
  = -D_h(L_h^n,\ell_h^n). \notag
\end{align}
\end{lemma}
\begin{proof}
One can verify by direct computation that
\begin{align}
\bdt E_h(L_h^n,\ell^n_h) 
&\le (\bdt L^n_h, \lambda (L^n_h-L^\infty_h))_{\Omega_h} + (\bdt \ell^n_h, \gamma ( \ell^n_h - \ell^\infty_h))_{\Gamma_h}.
\end{align}
The assertion then follows along the lines of the proof of  Lemma~\ref{lem:disentropy}.
\end{proof}


The exponential convergence to equilibrium can now be established in a similar manner as on the semi-discrete level by using the previous lemma, the entropy-entropy dissipation inequality of Lemma~\ref{lem:diseed}, and a discrete Gronwall inequality. Summarizing, we obtain 

\begin{theorem}[Convergence to the discrete equilibrium]\label{thrm:fullconverge}
For any $\tau >0$ and $n \ge 0$, there holds
\begin{align}\label{converge-full}
\|L^n_h - L^\infty_h\|_{\Omega_h}^2 + \|\ell^n_h - \ell^\infty_h\|_{\Gamma_h}^2 
\le C e^{-c_0 \tau n} (\|L^0_h - L^\infty_h\|_{\Omega_h}^2 + \|\ell^0_h - \ell^\infty_h\|_{\Gamma_h}^2).
\end{align}
The constants $C$ and $c_0$ can be chosen the same as in Theorem~\ref{thrm:disconverge}.
\end{theorem}


Let us emphasize that the constants are independent of the discretization parameters and the time horizon. This will allow us to establish uniform bounds for the discretization error.

\subsection{Error estimates for the full discretization}

For the error analysis on the fully discrete level, we again slightly extend standard arguments \cite{Thomee,Wheeler73} by appropriately taking into account the geometric errors. 
Our main error estimate for the full discretization reads

\begin{theorem}[Error estimates for the full discretization]\label{thrm:fullconvergence}
Let  (A1)--(A3) hold and assume that the solution $(L,\ell)$ of \eqref{eq:3a}--\eqref{eq:3d} is sufficiently smooth. 
Then for all $n\geq 0$ we have
\begin{align*}
\|L^n_h - \widetilde L(t^n)\|_{\Omega_h} + \|\ell^n_h - \widetilde \ell(t^n)\|_{\Gamma_h} \leq C(h^2+\tau)
\end{align*}
with a constant $C$ that is independent of the discretization parameters $\tau$ and $h$ and of $n$.
\end{theorem}
\begin{proof}
For ease of notation, we again write $U_h^n = (L^n_h,\ell^n_h)$, $\widetilde U = (\widetilde L,\widetilde{\ell})$, and $R_h U = (\widetilde L_R, \widetilde \ell_R)$.
Using the elliptic projection, the error can be decomposed into
\begin{align*}
U^n_h - \widetilde  U(t^n) = [U^n_h - R_h  U(t^n)] + [R_h  U(t^n) - \widetilde  U(t^n)] =: \theta_h^n + \rho^n.
\end{align*}
By the properties of the operator $R_h$, the first error component
$\rho^n$ can be estimated by
\begin{align}\label{rhon}
\|\rho^n\|_{L^2(\Omega_h) \times L^2(\Gamma_h)} \leq Ch^2\|U(t^n)\|_{H^2(\Omega)\times H^2(\Gamma)}.
\end{align}
Similar as in the proof of Theorem~\ref{thrm:convergencerate}, we 
can characterize the second error component by
\begin{align*}
&c_h(\bdt \theta_h^n, \Phi_h) +a_h(\theta_h^n;\Phi_h)\\
& \qquad = c_h(\bdt U^n_h, \Phi_h) - c_h(\bdt R_h U(t^n), \Phi_h) + a_h(U^n_h; \Phi_h) - a_h(R_h U(t^n); \Phi_h)\\
&\qquad = [\widetilde c(\bdt \widetilde U(t^n),\Phi_h) - c_h(\bdt \widetilde U(t^n);\Phi_h)]
+ c_h(\bdt \widetilde U(t^n) - R_h \bdt U(t^n); \Phi_h)  \\
&\qquad \qquad  \qquad \qquad \qquad \qquad 
+ \eta c_h(R_h U(t^n) - \widetilde U(t^n); \Phi_h) 
+ \widetilde c(\dt \widetilde U(t^n) - \bdt \widetilde U(t^n) ;\Phi_h) \\
& \qquad = (i)+(ii)+(iii)+(iv).
\end{align*}
The first three terms can be estimated as in Theorem~\ref{thrm:convergencerate},
and for the fourth term, we use 
\begin{align*}
\bdt \widetilde U(t^n) - \dt \widetilde U(t^n) = -\int_{t^{n-1}}^{t^n} \int_{t}^{t^n} U_{tt}(s) ds \; dt.
\end{align*}
Proceeding similarly as in the proof of Theorem~\ref{thrm:convergencerate},
we then arrive at
\begin{align*}
&\bdt E_h(\theta_h^n) + c_1 E_h(\theta_h^n)\\
&\qquad \leq C' \big(h^4 \|\dt U(t^n)\|_{H^2(\Omega \times H^2(\Omega)}^2 + h^4 \|U(t^n)\|_{H^2(\Omega) \times H^2(\Omega)}^2 + \tau^2 \|U_{tt}(\xi^n)\|_{L^2(\Omega) \times L^2(\Gamma)}^2 \big)
\end{align*}
for some positive constants $c_1,C'>0$ and appropriate $\xi^n \in (t^{n-1},t^n)$. 
The result now follows by a discrete Gronwall lemma.
\end{proof}

Let us emphasize once more that all error estimates hold uniformly in time.
By simple inspection of the proof, the precise regularity requirements can be made explicit. 

\section{Numerics} \label{sec:num}

For illustration of our theoretical results, we now report about some numerical results
obtained for the model problem \eqref{eq:1a}--\eqref{eq:1c}.
For ease of presentation, we restrict ourselves to two space dimensions 
and consider $\Omega$ to be the unit circle. 
In the simulations, we chose the model parameters as $d_L = 0.01$, $d_\ell = 0.02$, $\gamma = 2$, and $\lambda = 4$, and we set the initial conditions to 
\begin{align*}
 L^0(x,y) = \frac{1}{2}(x^2 + y^2) \quad \text{and} \quad \ell^0(x,y) = \frac{1}{2} (1+x).
\end{align*}
We start by investigating the convergence of solutions with respect to the mesh size. 
To this end, we generate a sequence of meshes $\{T_h\}_{h>0}$ by uniform refinement of an initial triangulation consisting of $258$ elements. 
We further choose $T=2$ as the final time and $\tau=0.01$ as the time step size. 
As a computable approximation for the error at time $t=T$, we use the difference of solutions obtained on two consecutive refinements. To be precise, we use 
$$
L(T) - L^N_h(T) \approx L^N_{h/2} - L_h^N := \triangle L_h
$$  
and $\triangle \ell_h=\ell_{h/2}^N - \ell_h^N$ to measure the error for the error in the volume and surface concentrations, respectively.
The results obtained in our numerical tests are displayed in Table \ref{tab:errorh}.
%
\begin{table}[ht!]\label{tab:errorh}
 \begin{tabular}{c||c|c||c|c||c|c||c|c}
 $h$ & $ \|\triangle L_h\|_{L^2}$ & rate & $\|\triangle \ell_h\|_{L^2}$ & rate & $\|\triangle L_h\|_{H^1}$ & rate & $\|\triangle \ell_h\|_{H^1}$ & rate \\ 
\hline 
0.2500 & 0.0066 & ---  & 0.0029 &  --- & 0.1193 &  --- & 0.0390 &  --- \\
0.1250 & 0.0018 & 1.88 & 0.0007 & 1.98 & 0.0608 & 0.97 & 0.0195 & 1.00 \\
0.0625 & 0.0005 & 1.96 & 0.0002 & 1.99 & 0.0306 & 0.99 & 0.0097 & 1.00 \\
0.0313 & 0.0001 & 1.99 & 0.0000 & 2.01 & 0.0154 & 1.00 & 0.0049 & 1.00 \\
0.0016 & 0.0000 & 2.00 & 0.0000 & 1.94 & 0.0077 & 1.00 & 0.0024 & 1.00 \\
 \end{tabular}
\vskip0.5em
\caption{\small Errors obtained on a sequence of uniformly refined meshes for the model problem \eqref{eq:1a}--\eqref{eq:1c} with convergence rate estimated from two consecutive levels.}
\end{table}
%
%
They confirm the convergence rates that were predicted by Theorem \ref{thrm:convergencerate}. 
Let us mention that the $O(h)$ convergence of the error in the $H^1$-norm can also be explained theoretically \cite{Thomee,Wheeler73}. 

\medskip 

In a second test, we verify the convergence with respect to the time step size $\tau$. 
For that we use a spatially refined mesh $T_h$ with mesh size $h=0.03125$. As before, we set the final time to $T=2$ and we run the simulations for 
different time steps sizes $\tau=2^{-k}$, $k=1,2,\ldots$. 
The results obtained in our simulations are depicted in Table~\ref{tab:errort}.

\begin{table}[ht]\label{tab:errort}
 \begin{tabular}{c||c|c||c|c||c|c||c|c}
 $\tau$ & $ \|\triangle L_h\|_{L^2}$ & rate & $\|\triangle \ell_h\|_{L^2}$ & rate & $\|\triangle L_h\|_{H^1}$ & rate & $\|\triangle \ell_h\|_{H^1}$ & rate \\ 
\hline 
0.5000 & 0.0082 & ---  & 0.0109 &  --- & 0.0464 &  --- & 0.0154 &  --- \\
0.2500 & 0.0046 & 0.82 & 0.0057 & 0.94 & 0.0269 & 0.79 & 0.0080 & 0.94 \\
0.1250 & 0.0025 & 0.90 & 0.0029 & 0.97 & 0.0147 & 0.88 & 0.0041 & 0.97 \\
0.0625 & 0.0013 & 0.95 & 0.0015 & 0.98 & 0.0077 & 0.93 & 0.0021 & 0.98 \\
0.0313 & 0.0007 & 0.97 & 0.0007 & 0.99 & 0.0039 & 0.96 & 0.0010 & 0.99 \\
 \end{tabular}
\vskip0.5em
\caption{\small Errors and convergence rates obtained with decreasing time step size $\tau$ for the model problem \eqref{eq:1a}--\eqref{eq:1c} with rate computed from two consecutive levels.}
\end{table}
 
The error history clearly indicates convergence of first order with respect to $\tau$ which agrees with the rate predicted by Theorem~\ref{thrm:fullconvergence}. 
Note that the same convergence rate is observed for the error in the $L^2$- and the $H^1$-norm. Also this observation could be proven theoretically. 

\medskip 

With the third test, we would like to illustrate the large time behavior of the system. 
To this end, we choose $T=500$ and $\tau = 0.5$.
We then compute the discrete solutions for a sequence of uniformly refined meshes and evaluate the entropy, to be more precise, we compute 
$$
\widetilde E(L_h^N,\ell_h^N) = \frac{1}{2} \big(\lambda \|L_h^N - L^\infty\|_{\Omega_h}^2 + \gamma \|\ell^N_h - \ell^\infty\|^2_{\Gamma_h}\big),
$$
which measures the distance of the discrete solution to the \emph{exact} equilibrium state.
This allows us to evaluate at the same time the convergence to equilibrium and the 
inexact approximation of the equilibrium state which is due to the geometric errors.
The corresponding results are  displayed in Figure \ref{fig:exponential}. 
\begin{figure}[ht!]
 \includegraphics[width=0.75\textwidth]{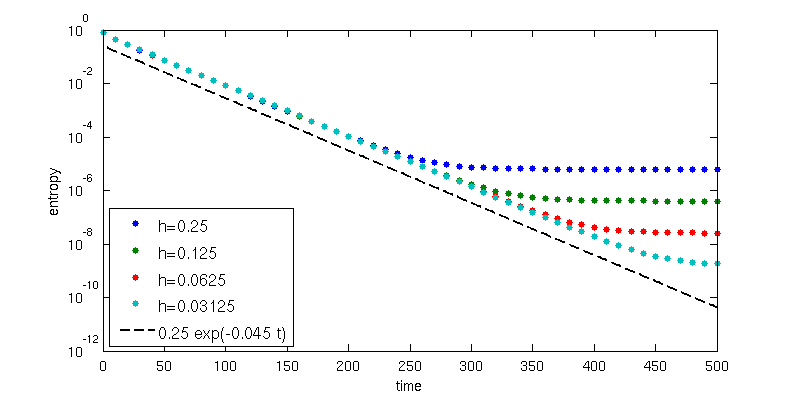}
\caption{\small Decay of the entropy $\widetilde E(L_h(t^n),\ell_h(t^n))$ over time. 
One can clearly see the exponential convergence with a mesh independent rate. 
Note that for large time, the discretization error for the equilibrium state becomes dominant
and leads to saturation.
}
\label{fig:exponential}
\end{figure}
As predicted by our theoretical results, we observe uniform exponential convergence to equilibrium. 
The numerical results also allow us to estimate the exponential decay rate which is approximately $c_0=0.045$ here.
For large time, the discretization errors due to approximation of the equilibrium becomes dominant; compare with Lemma~\ref{lem:eeeq}.

\section{Extension to a system arising in asymmetric stem cell division}\label{sec:extension}

As a final step of our considerations, we would like to illustrate now that the analysis and the discretization strategy presented for the model problem \eqref{eq:3a}--\eqref{eq:3c} can be extended with minor modifications to more general volume-surface reaction-diffusion systems which share the same key properties: (i) non-negativity of solutions, (ii) conservation of mass, (iii) a constant positive detailed balance equilibrium, and (iv) a quadratic entropy functional and an appropriate entropy-entropy dissipation estimate.

As an example, we consider the following system describing the evolution of two volume and two surface concentrations, which was already mentioned briefly in the introduction.
\begin{subequations}\label{eq:6}
\begin{align}
\dt L - d_L\Delta L = -\beta L + \alpha P, &\qquad x\in\Omega,  \ t>0,\label{eq:6a}\\
\dt P - d_P\Delta P = \beta L - \alpha P, &\qquad x\in\Omega,  \ t>0,\label{eq:6b}\\
\dt \ell - d_{\ell}\Delta_{\Gamma}\ell = -d_L\partial_n L + \chi_{\Gamma_2}(-\sigma \ell + \kappa p), &\qquad x\in\Gamma,\  t>0, \label{eq:6c}\\
\dt p - d_p\Delta_{\Gamma_2} = \sigma \ell - \kappa p - d_P\partial_n P, &\qquad x\in\Gamma_2, \ t>0.\label{eq:6d}
\end{align}
As before, $\Omega$ is a bounded domain in two or three dimensions with a smooth boundary $\Gamma = \partial\Omega$. 
Furthermore, $\Gamma_2 \subset \Gamma$ is a proper subpart with smooth boundary $\partial\Gamma_2$. We denote by $\chi_{\Gamma_2}$ the characteristic function on $\Gamma_2$. 
The reactions and the mass transfer between the volume and the surface are governed by 
\begin{align}
d_L\partial_n L = -\lambda L + \gamma \ell, &\qquad x\in\Gamma, \ t>0, \label{eq:6e}\\
d_P\partial_n P = \chi_{\Gamma_2}(-\eta P + \xi p), &\qquad x\in\Gamma,\ t>0, \label{eq:6f}\\
d_p\partial_{n_{\Gamma}}p = 0, &\qquad x\in\partial\Gamma_2, \ t>0,\label{eq:6g}
\end{align}
\end{subequations}
and the system is complemented by appropriate initial conditions. 
A schematic sketch of the reaction dynamics is given in Figure \ref{fig:dynamics}. 

A variant of the system \eqref{eq:6a}--\eqref{eq:6g} 
was studied recently in \cite{FRT} as a model for the asymmetric localization of Lgl during the mitosis of SOP stem cells of Drosophila \cite{BMK,MEBWK,WNK}.
The volume concentrations $L, P$ then represent the cytoplasmic pure and phosphorylated form of the protein Lgl (Lethal giant larvae) which is essential in triggering the asymmetric localization of so-called cell-fate determinant proteins during asymmetric division. Moreover, $\ell,p$ represent the corresponding cortical Lgl concentration. A detailed discussion can be found in \cite{FRT}.

We assume that all model parameters are positive constants and that 
the system has a \emph{detailed balance equilibrium}, i.e. we require that
\begin{align}\label{detailedbalance}
\frac{\alpha \lambda \sigma \xi}{\beta \gamma \kappa \eta} = 1.
\end{align}
Using this condition, one can show that the system \eqref{eq:6a}--\eqref{eq:6g} has very similar properties as the model problem \eqref{eq:3a}--\eqref{eq:3c} and, therefore, the analysis of the previous sections can be carried over to the system \eqref{eq:6} almost verbatim.
Let us sketch the necessary key steps in more detail:

\smallskip \noindent 
{\bf Step 1.} The system \eqref{eq:6a}--\eqref{eq:6g} has an inherent {\it mass conservation law}, i.e., 
\begin{align} \label{4x4mass}
M(t) := \int_{\Omega} L(t) + P(t) dx + \int_{\Gamma} \ell(t) ds + \int_{\Gamma_2} p(t) ds = M(0)
\qquad \text{for all } t > 0.
\end{align}

\smallskip \noindent 
{\bf Step 2.} Together with the detailed balance condition \eqref{detailedbalance}, 
one can show that for any initial mass $M^0>0$ there exists a unique positive constant equilibrium $(L^{\infty}, P^{\infty}, \ell^{\infty}, p^{\infty})$. Again, analytic formulas depending only on the initial mass $M^0$, on the model parameters, and on the geometry can be derived.

\smallskip \noindent 
{\bf Step 3.} We can define a {quadratic relative entropy functional},
which here has the form
\begin{align} \label{4x4entropy}
E(L,P,\ell,p)= \frac{1}{2}&\left(\int_{\Omega}\frac{1}{L^{\infty}}|L-L^{\infty}|^2dx + \int_{\Omega}\frac{1}{P^{\infty}}|P-P^{\infty}|^2dx \right.\nonumber\\
&\left.\qquad \qquad + \int_{\Gamma}\frac{1}{\ell^{\infty}}|\ell-\ell^{\infty}|^2 ds + \int_{\Gamma_2}\frac{1}{p^{\infty}}|p -p^{\infty}|^2 ds\right).
\end{align}
Let us mention that, up to scaling with a constant, also the entropy functional for the model problem \eqref{eq:3a}--\eqref{eq:3c} defined in \eqref{eq:entropy} could be written in this way. 

\smallskip \noindent 
{\bf Step 4.} The corresponding {entropy dissipation relation} now reads
\begin{align*}
&\tfrac{d}{dt}E(L(t),P(t),\ell(t),p(t)) \\
&= -\tfrac{d_L}{\Lin} \|\nabla L(t)\|^2_\Omega  
   - \tfrac{d_P}{\Pin} \|\nabla P(t)\|_{\Omega}
- \tfrac{d_{\ell}}{\ell^{\infty}} \|\nabla_{\Gamma}\ell(t)|^2_\Gamma 
- \tfrac{d_{p}}{p^{\infty}} \|\nabla_{\Gamma_2}p(t)\|_{\Gamma_2}^2\\
&\quad 
- \tfrac{1}{\beta L^{\infty}} \|\beta L(t) - \alpha P(t)\|_\Omega^2  - \tfrac{1}{\gamma\ell^{\infty}} \|\gamma \ell(t) - \lambda L(t))^2 \|_\Gamma^2
\\ &\quad 
- \tfrac{1}{\kappa p^{\infty}} \|\kappa p(t) - \sigma \ell(t)\|_{\Gamma_2}^2 
 - \tfrac{1}{\eta P^{\infty}} \|\eta P(t) - \xi p(t)\|_{\Gamma_2}^2 =: -D(L(t),P(t),\ell(t),p(t)).
\end{align*}

\smallskip \noindent 
{\bf Step 5.} Similar as in Lemma \ref{lem:poincare}, 
an {entropy-entropy dissipation estimate} of the form
\begin{align}\label{4x4eede}
D(L,P,\ell,p) \geq c_0\, E(L,P,\ell,p)
\end{align}
holds with a constant $c_0$ only depending on the model parameters and on the domain. 
Similar as above, the proof of this estimate is again based on a Poincar\'e-type inequality.

\bigskip 

\noindent
Following the arguments of Sections~\ref{sec:analysis}--\ref{sec:time} 
one can then establish the following results:

\smallskip \noindent 
{\bf Result 1.} 
Convergence to equilibrium for the continuous problem, i.e., 
\begin{align*}
\|L(t) - \Lin\|_{\Omega} + \|P(t) - \Pin\|_{\Omega} + \|\ell(t) - \llin\|_{\Gamma} + \|p(t) - \ppin\|_{\Gamma_2} \leq Ce^{-c_0t} \; \text{ for all } \; t>0.
\end{align*}

\smallskip \noindent 
{\bf Result 2.} 
Convergence to the discrete equilibrium for semi-discrete solutions
\begin{align*}
\|L_h(t) - \Lin_h\|_{\Omega_h} + \|P_h(t) - \Pin_h\|_{\Omega_h} + \|\ell_h(t) - \llin_h\|_{\Gamma_h} + \|p_h(t) - \ppin_h\|_{\Gamma_{2,h}} \leq Ce^{-c_0t},
\end{align*}
with constants $c_0,C$ independent of the mesh size $h>0$

\smallskip \noindent
{\bf Result 3.} 
Convergence for the fully discrete solution, i.e., 
\begin{align*}
\|L_h^n - \Lin_h\|_{\Omega_h} + \|P_h^n - \Pin_h\|_{\Omega_h} + \|\ell_h^n - \llin_h\|_{\Gamma_h} + \|p_h^n - \ppin_h\|_{\Gamma_{2,h}} \leq C e^{-c_0 \tau n},
\end{align*}
with constants $c_0,C$ independent of the mesh size $h>0$ and the time step $\tau>0$.

\smallskip \noindent 
{\bf Result 4.} 
Error estimates and order optimal convergence for the semi-discretization independent of time horizon
\begin{align*}
\|L_h(t) - \widetilde L(t)\|_{\Omega_h} + \|P_h(t) - \widetilde P(t)\|_{\Omega_h} + \|\ell_h(t) -\widetilde\ell(t)\|_{\Gamma_h} + \|p_h(t)-\widetilde p(t)\|_{\Gamma_{2,h}}\leq Ch^2.
\end{align*}

\smallskip \noindent 
{\bf Result 5.} 
Order optimal error estimates for the full discretization
\begin{align*}
\|L_h^n - \widetilde L(t^n)\|_{\Omega_h} + \|P_h^n -\widetilde P(t^n)\|_{\Omega_h} + \|\ell_h^n - \widetilde \ell(t^n)\|_{\Gamma_h} + \|p_h^n-\widetilde p(t^n)\|_{\Gamma_{2,h}}\leq C(h^2+\tau)
\end{align*}
with constant $C$ that is independent of the time horizon, of the mesh size, and of the time step.

\medskip

\noindent 
{\bf Conclusion:} 
All results obtained in the previous sections for the model problem \eqref{eq:3a}--\eqref{eq:3d} carry over almost verbatim to more complex problems
sharing the same key properties.

\bigskip 

For illustration of the solution behavior, 
some snapshots of the concentrations $P(t)$ and $L(t)$ for a typical 
scenario are depicted in Figure~\ref{fig:evolution_44}.

\begin{figure}[ht!]\label{fig:evolution_44}
	\includegraphics[width=0.24\textwidth]{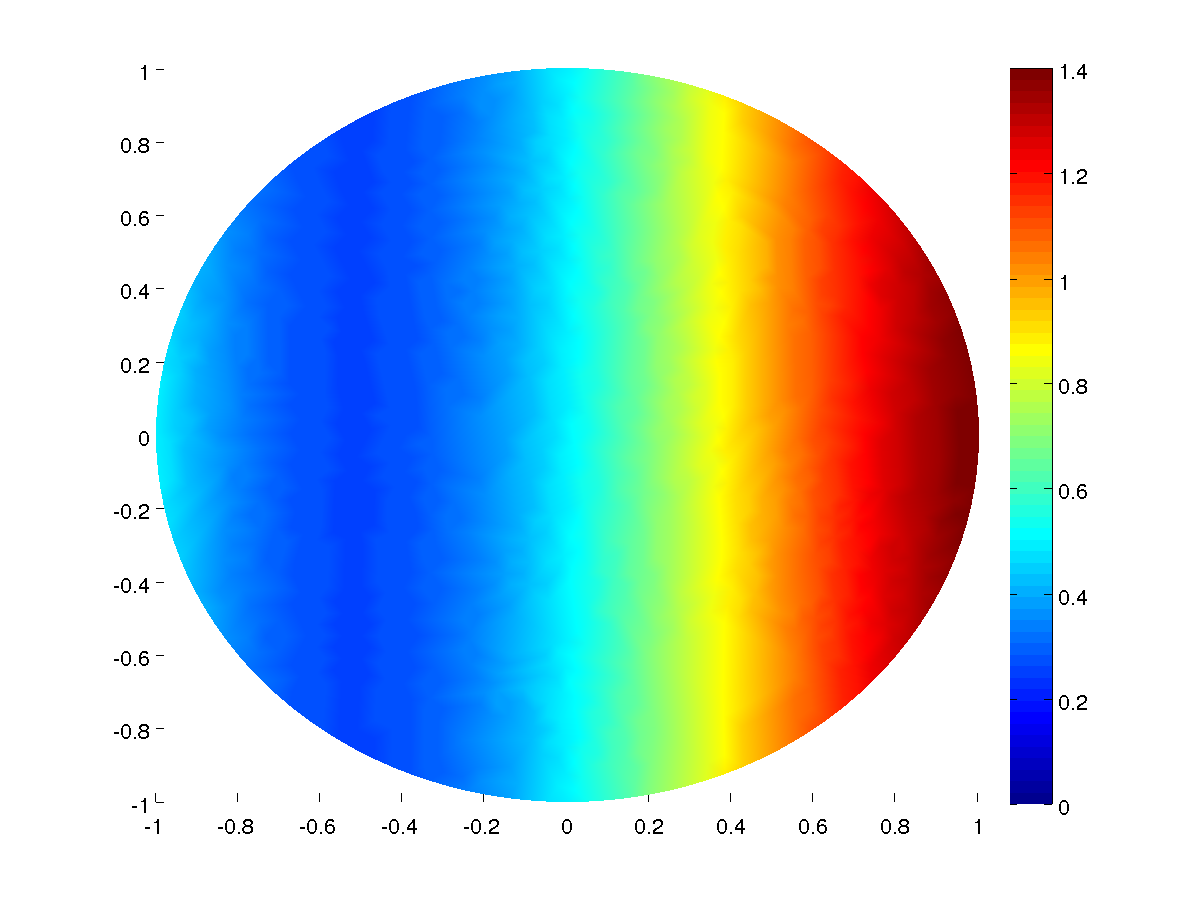}\includegraphics[width=0.24\textwidth]{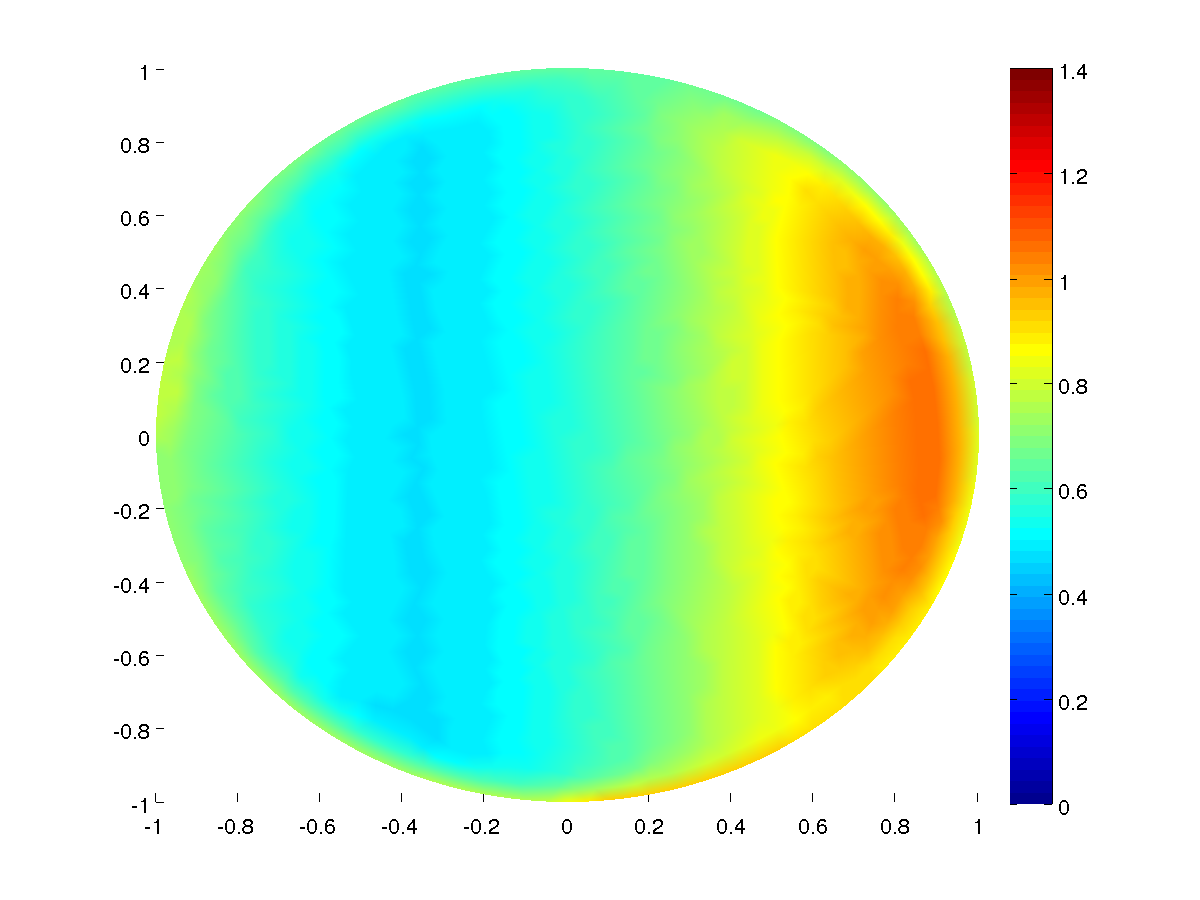}\includegraphics[width=0.24\textwidth]{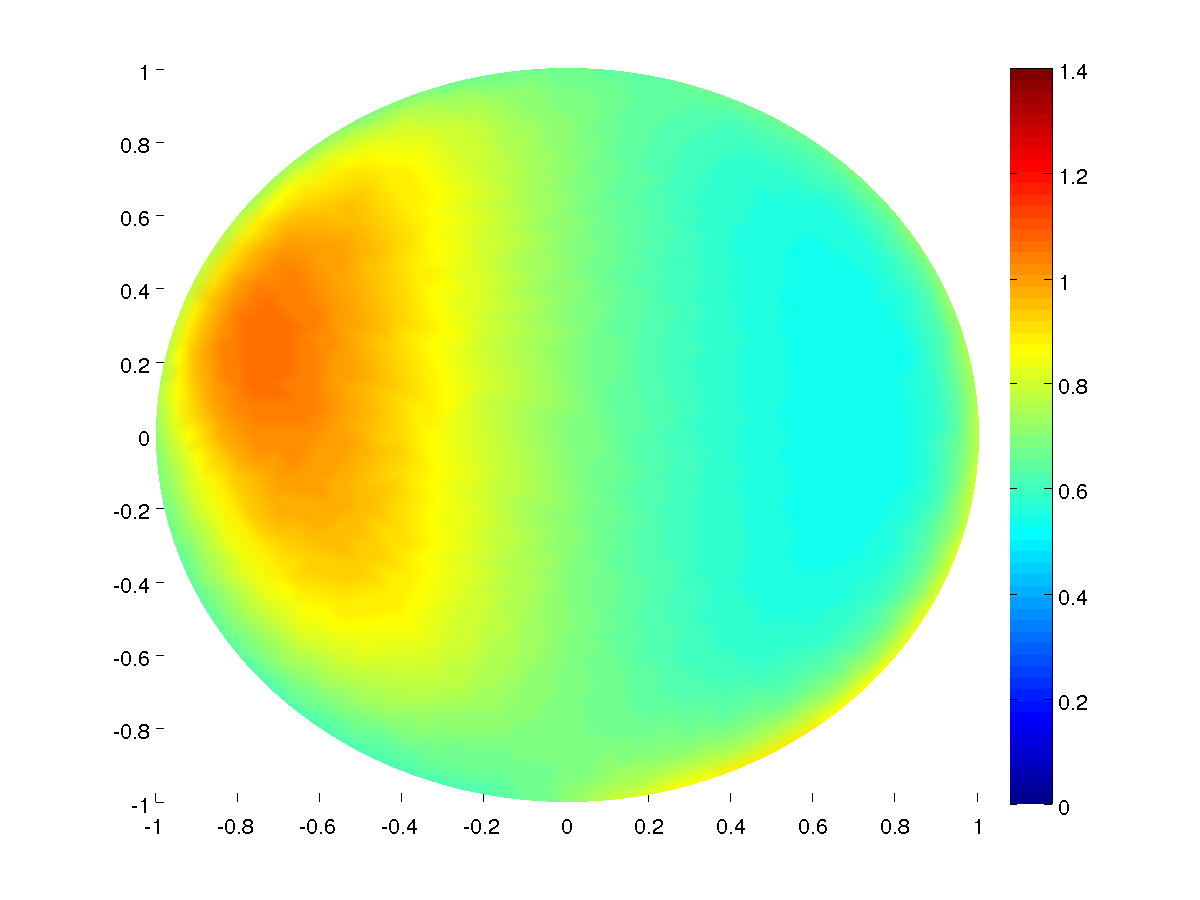}\includegraphics[width=0.24\textwidth]{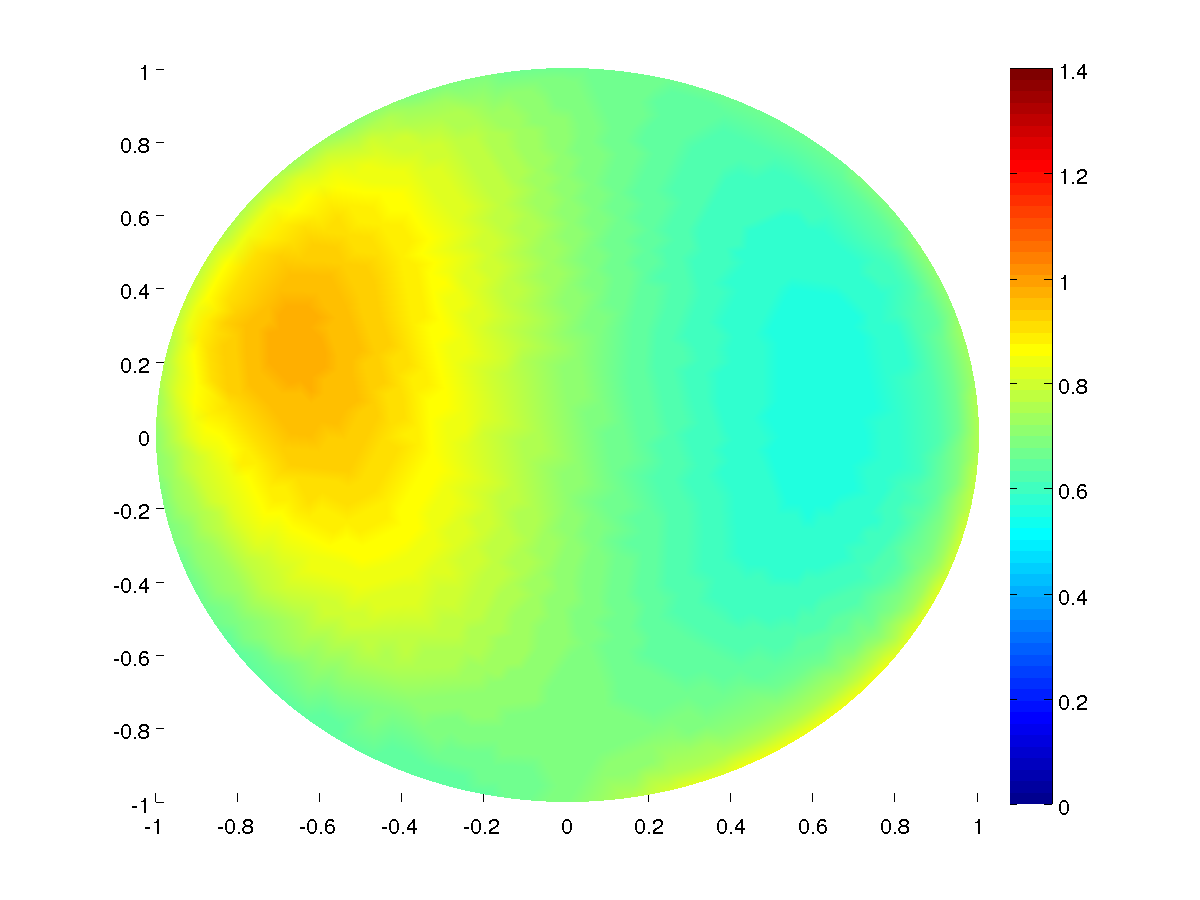}\\
	\includegraphics[width=0.24\textwidth]{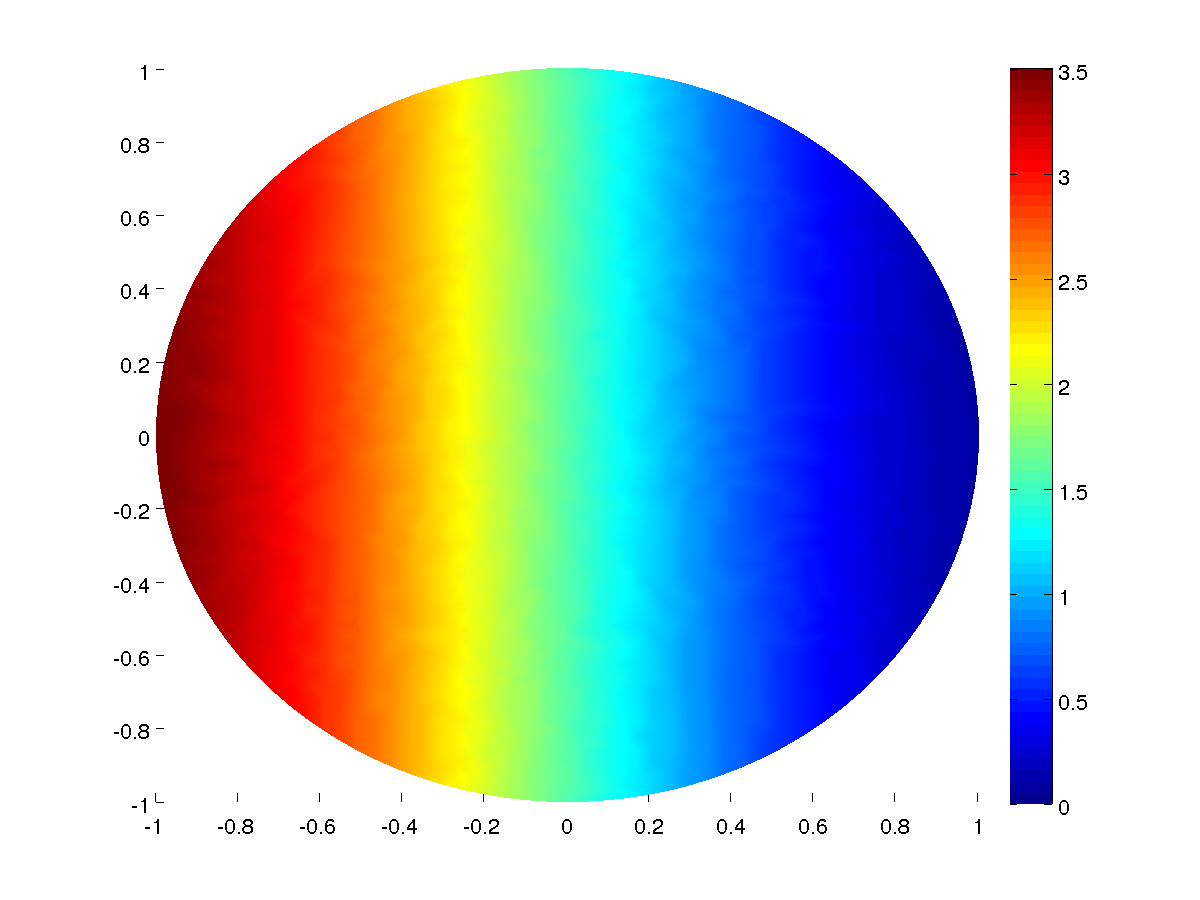}\includegraphics[width=0.24\textwidth]{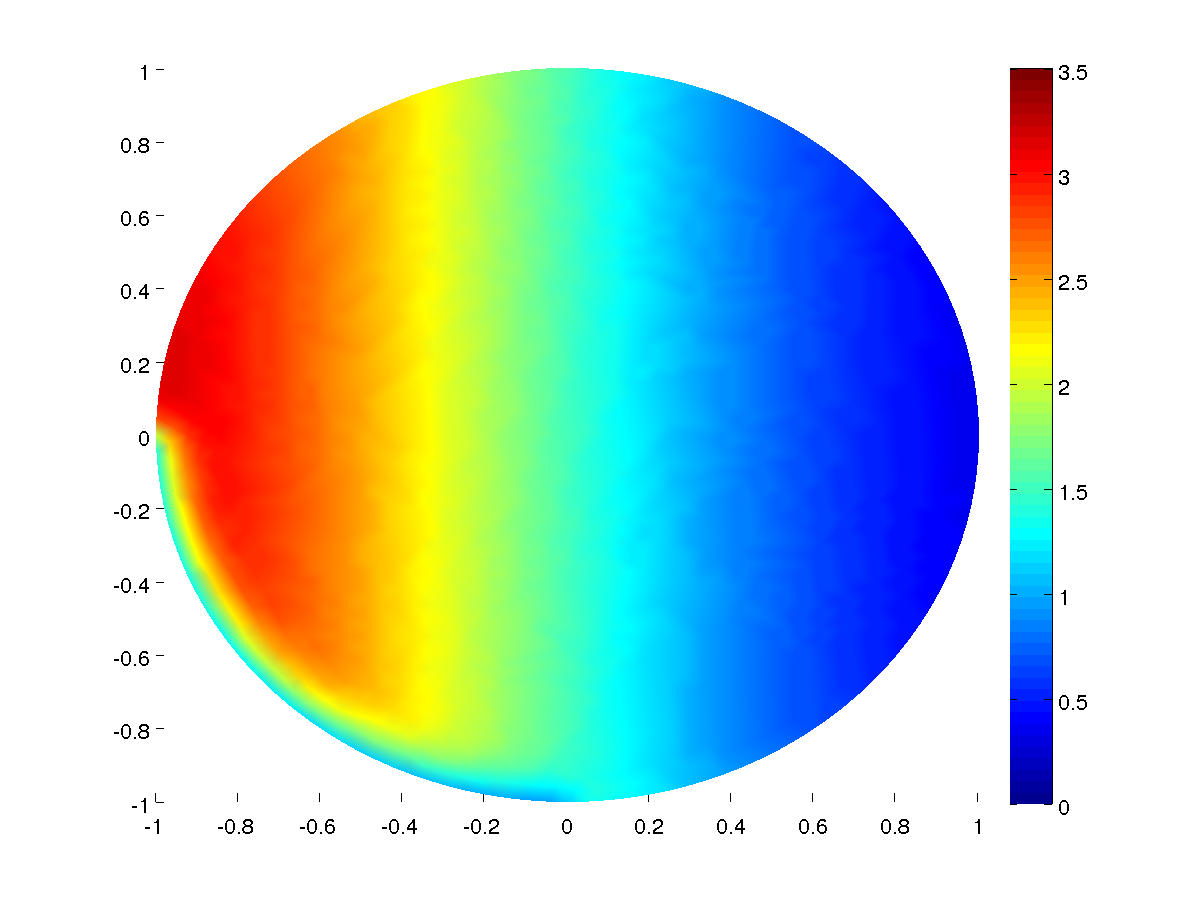}\includegraphics[width=0.24\textwidth]{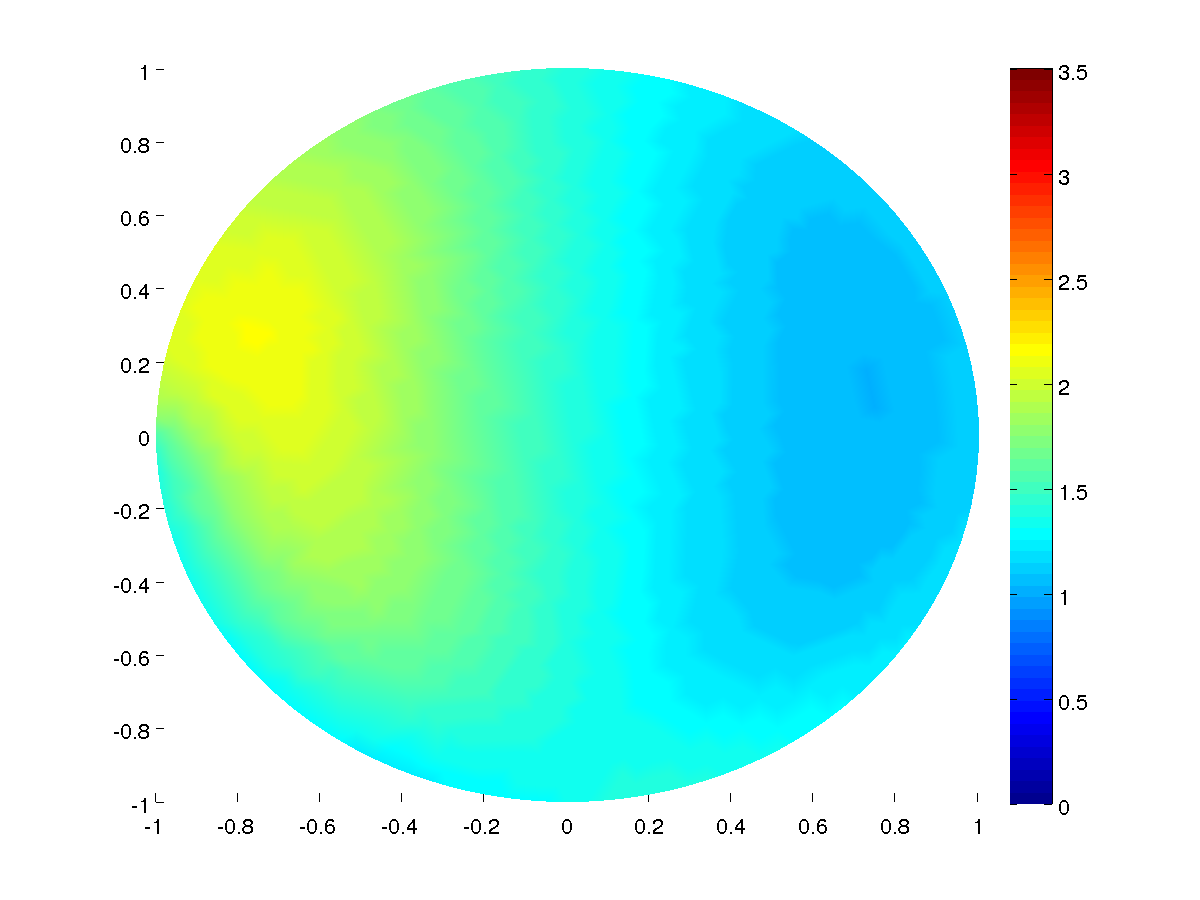}\includegraphics[width=0.24\textwidth]{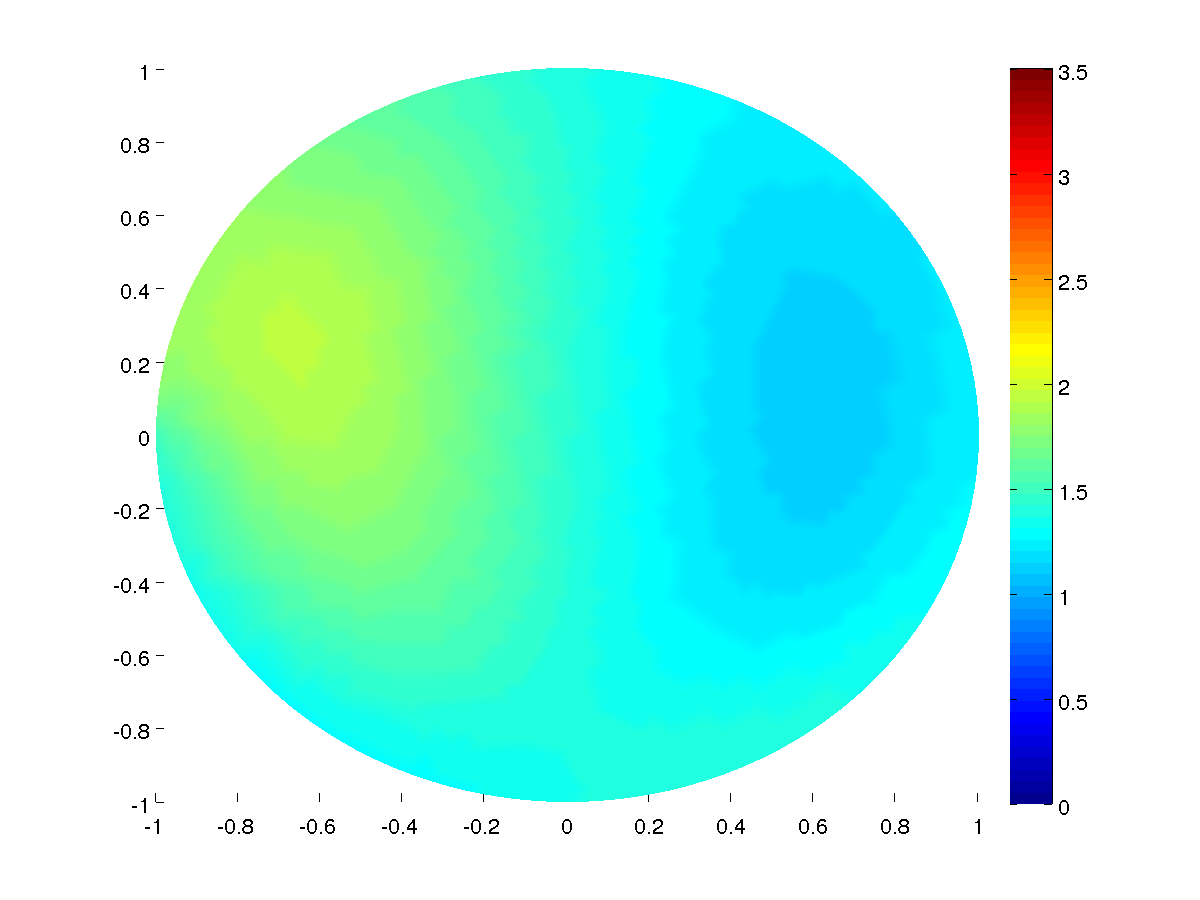}
	\caption{\small Snapshots for $L(t)$ and $P(t)$ at $t=0,0.13,1.56,3.0$. A mesh with $4064$ triangles and $\tau = 0.01$ was used. The initial data were $L^0(x,y) = x\sin(x+1)+0.5,\, P^0(x,y) = (2-x)\cos(x+1)+0.5,\,\ell^0(x,y) = 0.3(2-y)+1,\,\text{and }p^0(x,y) = 0.4y+1$.}
\end{figure}

\noindent 
As can be seen from the images, the evolution is driven by convergence to the constant equilibrium. 
Also some local effects due to the mass transfer with the boundary can be observed.

\section{Summary} \label{sec:discussion}

Volume-surface reaction-diffusion systems arise in many applications
in chemistry, in fluid dynamics, in crystal growth \cite{GG,NGCRSS},
and also in molecular-biology, where many current models aim to describe signaling pathways \cite{FNR} or the evolution of proteins \cite{FRT}.
Realistic models often involve several species and the mathematical analysis becomes increasingly  cumbersome for larger systems. The guideline of this paper therefore was to develop analytical and numerical methods which are general in the sense that they are based only on a few fundamental properties which are shared by many problems, e.g., 
the non-negativity of solutions, constant equilibria, exponential stability and convergence to equilibrium.  


Although we confined ourselves here to a simple model problem, our arguments, 
in particular the use of entropy estimates for the analysis on the continuous and 
on the discrete level, can be used to obtain similar results for a large class of surface-volume reaction-diffusion problems having constant equilibrium states. 
Big parts of our analysis could even be extended to problems with non-constant equilibria. 
In particular, quadratic entropy functionals of the form \eqref{4x4entropy} can be constructed for a wide class of problems with linear reaction dynamics. 
In the general situation, the entropies can however not be used 
so easily on the discrete level. 
Another direction of further research are complex balanced systems, 
e.g., weakly reversible reaction networks.
Such systems still feature positive equilibria and exponential convergence to equilibrium, as recently established for general first order reaction-diffusion networks in \cite{FPT}.
The formulation of an appropriate discrete entropy functional and the 
corresponding entropy dissipation estimate however require some non-trivial modifications 
of the arguments and are topics of current research.

\vskip 0.5cm
\noindent{\bf Acknowledgments.} 
This work was partly carried out during the visits of the second and the last author to TU Darmstadt.
KF and BT gratefully acknowledge the hospitality of TU Darmstadt and financial support by IRTG~1529, IRTG~1754, and NAWI Graz. 
HE was supported by DFG via grants IRTG~1529, GSC~233, and TRR~154. 
The work of JFP was supported by DFG via Grant 1073/1-1, by the Daimler and Benz Stiftung via Post-Doc Stipend 32-09/12 and by the German Academic Exchange Service via PPP grant No. 56052884. 

\def\underline#1{\emph{#1}}

\end{document}